\documentclass[11pt]{amsart}
\usepackage[margin=1in]{geometry}

\usepackage{amssymb}
\usepackage{amsthm}
\usepackage{amsmath}
\usepackage{mathrsfs}
\usepackage{amsbsy}
\usepackage[all]{xy}
\usepackage{bm}
\usepackage{hyperref}
\usepackage{tikz}
\usepackage{array}
\usepackage{float}
\usepackage{enumerate}
\usepackage{xcolor}
\usepackage{hhline}
\allowdisplaybreaks
\usepackage{cite}
\usepackage{tabularx,booktabs}

\newcommand{\SC}{\operatorname{SC}}
\newcommand{\Sort}{\operatorname{Sort}}
\newcommand{\Av}{\operatorname{Av}}
\newcommand{\rev}{\operatorname{rev}}
\newcommand{\comp}{\operatorname{comp}}

\newcommand{\swd}{\operatorname{swd}}
\newcommand{\sd}{\operatorname{sd}}

\newtheorem{theorem}{Theorem}[section]
\newtheorem{proposition}[theorem]{Proposition}
\newtheorem{corollary}[theorem]{Corollary}
\newtheorem{conjecture}[theorem]{Conjecture}

\newtheorem{lemma}[theorem]{Lemma}

\theoremstyle{definition}

\newtheorem{remark}[theorem]{Remark}
\newtheorem{example}[theorem]{Example}
\begin{document}

\title[]{Stack-Sorting with Consecutive-Pattern-Avoiding Stacks}
\subjclass[2010]{}

\author[]{Colin Defant}
\address[]{Fine Hall, 304 Washington Rd., Princeton, NJ 08544}
\email{cdefant@princeton.edu}
\author[]{Kai Zheng}
\address[]{Fine Hall, 304 Washington Rd., Princeton, NJ 08544}
\email{kzheng@princeton.edu}
\maketitle

\begin{abstract}
We introduce consecutive-pattern-avoiding stack-sorting maps $\SC_\sigma$, which are natural generalizations of West's stack-sorting map $s$ and natural analogues of the classical-pattern-avoiding stack-sorting maps $s_\sigma$ recently introduced by Cerbai, Claesson, and Ferrari. We characterize the patterns $\sigma$ such that $\Sort(\SC_\sigma)$, the set of permutations that are sortable via the map $s\circ\SC_\sigma$, is a permutation class, and we enumerate the sets $\Sort(\SC_{\sigma})$ for $\sigma\in\{123,132,321\}$. We also study the maps $\SC_\sigma$ from a dynamical point of view, characterizing the periodic points of $\SC_\sigma$ for all $\sigma\in S_3$ and computing $\max_{\pi\in S_n}|\SC_\sigma^{-1}(\pi)|$ for all $\sigma\in\{132,213,231,312\}$. In addition, we characterize the periodic points of the classical-pattern-avoiding stack-sorting map $s_{132}$, and we show that the maximum number of iterations of $s_{132}$ needed to send a permutation in $S_n$ to a periodic point is $n-1$. The paper ends with numerous open problems and conjectures. 
\end{abstract}

\section{Introduction}
\subsection{Background}

Let $S_n$ denote the set of permutations of the set $[n]:=\{1,\ldots,n\}$, which we write as words in one-line notation. The investigation of pattern avoidance in permutations began in 1968, when Knuth \cite{Knuth} introduced a stack-sorting machine and showed that a permutation can be sorted to the identity permutation using this machine if and only if it avoids the pattern $231$ (we define pattern containment and avoidance formally below). In his 1990 Ph.D. dissertation, West \cite{West} introduced a deterministic variant of Knuth's machine; this variant is a function $s:S_n\to S_n$. As in Knuth's case, a permutation $\pi\in S_n$ satisfies $s(\pi)=123\cdots n$ if and only if it avoids $231$. More recently, Albert, Homberger, Pantone, Shar, and Vatter \cite{AHPSV} introduced a vast generalization of Knuth's stack-sorting machine by considering $\mathcal C$-machines. Roughly speaking, a $\mathcal C$-machine is a sorting machine that uses a stack whose entries, when read from top to bottom, must have the same relative order as an element of the permutation class $\mathcal C$. Even more recently, Cerbai, Claesson, and Ferrari \cite{Cerbai} generalized West's stack-sorting map in a similar manner. For each permutation pattern $\sigma$, they defined a map $s_\sigma:S_n\to S_n$ that sends each permutation through a stack in a right-greedy manner, insisting that the contents of the stack must avoid the pattern $\sigma$ when read from top to bottom. In this more general setup, West's stack-sorting map is just $s_{21}$. Although the article \cite{Cerbai} is quite recent, it has already spawned several subsequent papers \cite{Baril, Berlow, Cerbai2, Cerbai3}. 

In this article, we introduce consecutive-pattern-avoiding stack-sorting maps $\SC_\sigma$. In order to define the maps $s_\sigma$ and $\SC_\sigma$ formally and simultaneously, it is helpful to define more general maps $T_{\mathcal A}$ dependent on sets $\mathcal A$ of permutations. Assume we are given an input permutation $\pi=\pi_1\cdots\pi_n$. Throughout this procedure, we consider the permutation obtained by reading the contents of a vertical stack from top to bottom. If moving the next entry in the input permutation to the top of the stack results in a stack whose contents have the same relative order as any element of $\mathcal A$, then the next entry in the input permutation is placed at the top of the stack (this operation is called a \emph{push}). Otherwise, the entry at the top of the stack is annexed to the end of the growing output permutation (this operation is called a \emph{pop}). This procedure stops when the output permutation has length $n$. We then define $T_{\mathcal A}(\pi)$ to be this output permutation. When $\mathcal A$ is the set of permutations that avoid $\sigma$, the map $T_{\mathcal A}$ is $s_\sigma$. Our focus in this article is on the case in which $\mathcal A$ is the set of permutations that avoid $\sigma$ as a consecutive pattern, in which case we write $T_{\mathcal A}=\SC_\sigma$. 

\begin{example}
The following diagram portrays the steps involved in sending the permutation $265413$ through a consecutive-$231$-avoiding stack, showing that $\SC_{231}(265413)=653142$. 
\begin{center}
\includegraphics[width=1\linewidth]{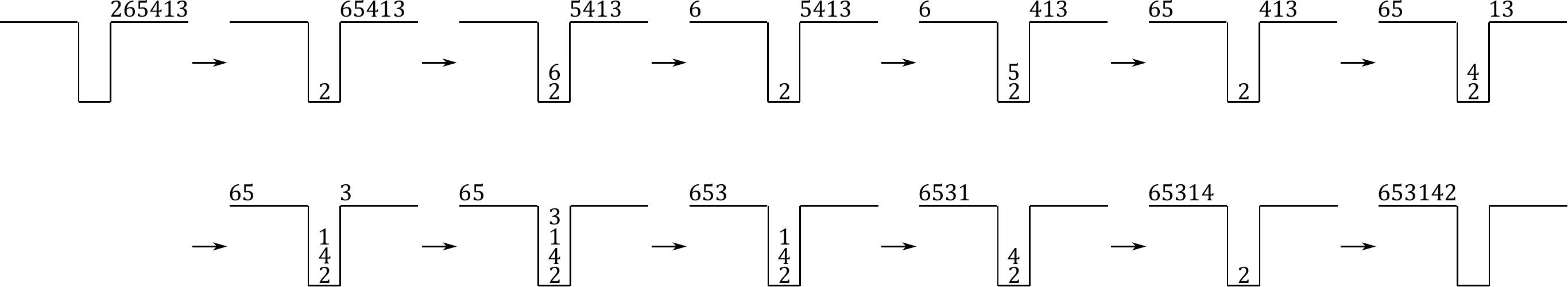}
\end{center}  
Notice that there is one step where the contents of the stack, when read from top to bottom, form the permutation $3142$. The permutation $3142$ contains the pattern $231$ classically (i.e., nonconsecutively), but avoids $231$ consecutively, which is why it is allowed to be in the stack. On the other hand, this would not be allowed if we were using a classical-$231$-avoiding stack. In that case, the steps would proceed as in the following diagram, which shows that $s_{231}(265413)=651432$. 

\begin{center}
\includegraphics[width=1\linewidth]{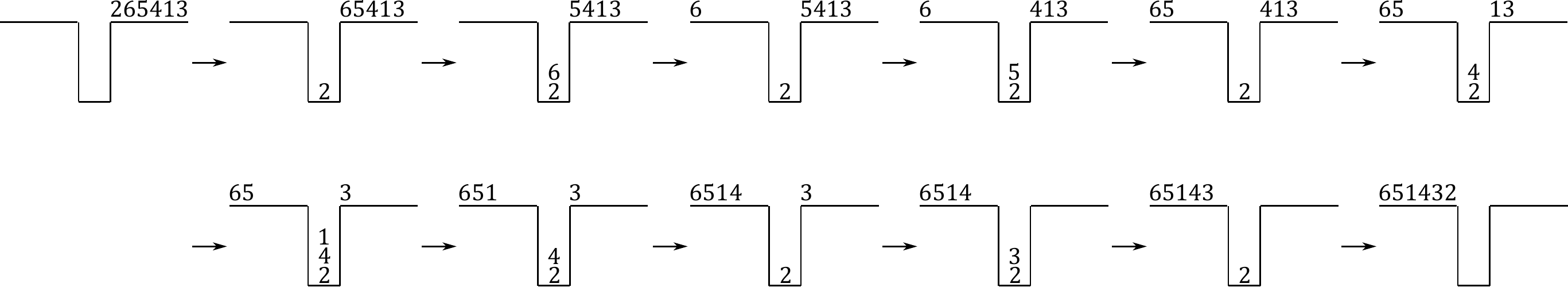}
\end{center}  

\end{example}

Our newly-defined maps $\SC_\sigma$ serve as perfectly legitimate generalizations of West's stack-sorting map because $\SC_{21}$ is the same as $s$, which is also the same as $s_{21}$ (a permutation avoids the pattern $21$ classically if and only if it avoids $21$ consecutively). However, when the length of $\sigma$ is at least $3$, the maps $\SC_\sigma$ and $s_\sigma$ exhibit different behaviors. 

The decision to consider stacks that avoid consecutive patterns, rather than stacks that avoid classical patterns, is very natural. Indeed, when one actually works through the process of sending a permutation through a classical-$21$-avoiding stack, one checks at each step whether the next entry in the input permutation is greater than the entry currently sitting at the top of the stack. One does not bother checking if the next entry in the input is greater than all of the entries in the stack because the smallest number in the stack is necessarily the top entry (this is another manifestation of the simple observation that a permutation avoids $21$ classically if and only if it avoids $21$ consecutively). If $\sigma$ has length $k\geq 3$, then when one sends a permutation through a classical-$\sigma$-avoiding stack, one must check all of the entries in the stack to see if the addition of the next entry produces an occurrence of the pattern $\sigma$. By contrast, when one sends a permutation through a consecutive-$\sigma$-avoiding stack, it is only necessary to compare the next entry in the input permutation with the top $k-1$ entries in the stack. In this way, the mechanics governing the maps $\SC_\sigma$ are in some sense ``closer'' to those governing West's stack-sorting map than those of the maps $s_\sigma$ are.  

Typically, researchers have approached West's stack-sorting map from a ``sorting'' point of view. The usual questions concern structural and enumerative aspects of $t$-stack-sortable permutations, which are permutations that get sorted to the identity using at most $t$ iterations of the map $s$ \cite{Bona, BonaSurvey,Bousquet98,Claessonn-4,DefantCounting,DefantPreimages,Goulden,West,Zeilberger}. On the other hand, there are now several articles that study the stack-sorting map from a ``dynamical'' point of view (see \cite{Bousquet, DefantCatalan, DefantEngenMiller, DefantFertility, DefantTroupes} and the references therein). In this approach, we view the stack-sorting map as an interesting combinatorially-defined function, without caring so much about trying to sort permutations. The typical questions concern the number of preimages of a permutation under $s$, which is called the \emph{fertility} of the permutation. The investigation of fertilities of permutations has revealed unexpected patterns and connections with several other aspects of combinatorics \cite{Bousquet, DefantCatalan, DefantFertility, DefantEngenMiller, Hanna, Maya}, 
including a very surprising and useful link with the combinatorics of cumulants in noncommutative probability theory \cite{DefantTroupes}. This dynamical approach concerning fertilities has also shed new light on the $t$-stack-sortable permutations studied from the sorting point of view \cite{DefantCounting,DefantPreimages}. 

In this paper, we consider the maps $\SC_\sigma$ from both the sorting point of view and the dynamical point of view. All of the sorting questions we consider are analogues of questions asked about the maps $s_\sigma$ in \cite{Cerbai,Cerbai2}. On the other hand, our dynamical approach is completely novel in the realm of pattern-avoiding stack-sorting maps (other than $s$ itself). We will also see that there are several dynamical questions that one can ask about the maps $\SC_\sigma$ that are trivial when $\sigma=21$ (i.e., when $\SC_\sigma=s$), but which become much more interesting when the length of $\sigma$ is at least $3$. 

\subsection{Outline and Summary of Main Results}

The articles \cite{Cerbai,Cerbai2} focus primarily on the set $\Sort(s_\sigma)$ of permutations that get sorted to the identity using a $\sigma$-avoiding stack followed by a classical-$21$-avoiding stack. Since the set of permutations that $s$ maps to the identity is the set $\Av(231)$ of $231$-avoiding permutations, we have $\Sort(s_\sigma)=s_\sigma^{-1}(\Av(231))$. Note that $\Sort(s_{21})=s^{-1}(\Av(231))$ consists of the $2$-stack-sortable permutations, which have received an enormous amount of attention since their inception in West's dissertation \cite{Bona, BonaSurvey,Bousquet98,DefantCounting,DefantPreimages,Goulden,West,Zeilberger}. Let us also remark that the articles \cite{Baril, Cerbai3} have studied other close variants of these sets; for example, the paper \cite{Baril} explores the sets $\Sort(s_{\sigma,\tau})=s_{\sigma,\tau}^{-1}(\Av(231))$, where $s_{\sigma,\tau}$ is the map that sends a permutation through a right-greedy stack whose contents avoid both $\sigma$ and $\tau$ (using our earlier notation, $s_{\sigma,\tau}=T_{\Av(\sigma,\tau)}$). In Section~\ref{Sec:CharacterizeClasses}, we define $\Sort(\SC_\sigma)=\SC_{\sigma}^{-1}(\Av(231))$, which can similarly be seen as the set of permutations that get sorted into the identity by a consecutive-$\sigma$-avoiding stack followed by a $21$-avoiding stack. We will prove that if $\sigma\in S_k$ for some $k\geq 3$, then $\Sort(\SC_\sigma)$ is a permutation class if and only if $\sigma$ and $\widehat\sigma$ both contain the pattern $231$, where $\widehat\sigma$ is the permutation obtained by swapping the first $2$ entries in $\sigma$; this is a direct analogue of one of the main results from \cite{Cerbai}. 

The other main results from \cite{Cerbai} concern the enumeration of the sets $\Sort(s_{123})$ and $\Sort(s_{321})$. Extending this work, the article \cite{Cerbai2} focuses on the enumeration of the set $\Sort(s_{132})$. In Sections~\ref{Sec:132and312} and \ref{Sec:123and321}, we give analogues of these results by characterizing and enumerating the sets $\Sort(\SC_{123})$, $\Sort(\SC_{321})$, and $\Sort(\SC_{132})$. More precisely, we will see that $\Sort(\SC_{123})$ is enumerated by the first differences of Motzkin numbers, that $\Sort(\SC_{321})$ is enumerated by the Motzkin numbers, and that $\Sort(\SC_{132})$ is enumerated by the Catalan numbers. 

In Section~\ref{Sec:PeriodicPoints}, we adopt the dynamical point of view and consider the periodic points of the maps $\SC_{\sigma}$, which we view as functions from $S_n$ to $S_n$ for each $n\geq 1$. It is easy to show that the only periodic point of $s:S_n\to S_n$ is the identity permutation $123\cdots n$. By contrast, if the length of $\sigma$ is at least $3$, then $\SC_\sigma$ has multiple periodic points. We will prove that if $\sigma\in S_3$, then the periodic points of $\SC_\sigma:S_n\to S_n$ are precisely the permuatations in $S_n$ that avoid $\sigma$ and its reverse as consecutive patterns.

It is known that the identity permutation $123\cdots n$, which has $C_n:=\frac{1}{n+1}{2n\choose n}$ preimages under $s$, has strictly more preimages under $s$ than every other permutation in $S_n$ (see \cite[Chapter 8, Exercise 23]{Bona}). The analogous results for the maps $\SC_\sigma$ are more difficult. In Section~\ref{Sec:132and312}, we study preimages under the maps $\SC_{132}$ and $\SC_{312}$, proving in particular that \[\max_{\pi\in S_n}|\SC_{132}^{-1}(\pi)|=\max_{\pi\in S_n}|\SC_{312}^{-1}(\pi)|={n-1\choose\left\lfloor\frac{n-1}{2}\right\rfloor}\] (along the way, we obtain a simple formula for $|\SC_{132}^{-1}(\pi)|$ for all reverse-layered permutations $\pi$). In Section~\ref{Sec:213and231}, we prove a similar result for the maps $\SC_{231}$ and $\SC_{213}$, showing that \[\max_{\pi\in S_n}|\SC_{231}^{-1}(\pi)|=\max_{\pi\in S_n}|\SC_{213}^{-1}(\pi)|=2^{n-2}\] when $n\geq 2$.

Section~\ref{Sec:Dynamicsofs132} is a little different from the rest of the paper because it focuses on classical-pattern-avoiding stack-sorting maps instead of consecutive-pattern-avoiding stack-sorting maps. The purpose of this section is to prove dynamical results about the maps $s_{132}$ and $s_{312}$. Such questions have not previously been considered for these maps, and we believe our results open the gate for interesting further developments. Specifically, we prove that the periodic points of $s_{132}$ (respectively, $s_{312}$) are precisely the permutations that avoid $132$ and $231$ (respectively, $312$ and $213$). We also prove the the maximum number of iterations needed for $s_{132}$ (respectively, $s_{312}$) to send a permutation to one of its periodic points is $n-1$ (these results are much more difficult than the analogous results for West's stack-sorting map). 

Finally, Section~\ref{Sec:Conclusion} collects a large number of open problems, conjectures, and other suggestions for future work. 

\subsection{Terminology and Notation}

Recall that in this article, a \emph{permutation} is an ordering of the elements of the set $[n]$ for some $n$. If $\pi$ is a sequence of $n$ distinct integers, then the \emph{standardization} of $\pi$ is the permutation in $S_n$ obtained by replacing the $i$th-smallest entry in $\pi$ with $i$ for all $i$. For example, the standardization of $4829$ is $2314$. We say two sequences have the \emph{same relative order} if their standardizations are equal. We say a permutation $\sigma$ \emph{contains} a permutation $\tau$ as a pattern if there is a (not necessarily consecutive) subsequence of $\sigma$ that has the same relative order as $\tau$; otherwise, $\sigma$ \emph{avoids} $\tau$. We say $\sigma$ \emph{contains} $\tau$ \emph{consecutively} if $\sigma$ has a consecutive subsequence with the same relative order as $\tau$; otherwise, $\sigma$ \emph{avoids} the consecutive pattern $\tau$. When we wish to stress that pattern containment (or avoidance) is nonconsecutive, we will sometimes call it \emph{classical} pattern containment (or avoidance). We refer to the standard references \cite{Bona, Elizalde, Kitaev, Linton} for more information about classical and consecutive permutation patterns. 

We denote consecutive patterns by underlining them (or referring to them as consecutive patterns). If $\tau^{(1)},\tau^{(2)},\ldots$ is a (finite or infinite) list of classical or consecutive patterns, then we write $\Av_n(\tau^{(1)},\tau^{(2)},\ldots)$ for the set of permutations in $S_n$ that avoid $\tau^{(1)},\tau^{(2)},\ldots$. We also write $\Av(\tau^{(1)},\tau^{(2)},\ldots)=\bigcup_{n\geq 0}\Av_n(\tau^{(1)},\tau^{(2)},\ldots)$. For example, $\Av(231,321)$ is the set of permutations that avoid $231$ and $321$, while $\Av(\underline{321})$ is the set of permutations that avoid the consecutive pattern $321$. If $\tau$ is an arbitrary unspecified permutation, then we let $\Av(\underline{\tau})$ be the set of permutations that consecutively avoid $\tau$. 

Given a permutation $\sigma=\sigma_1\cdots\sigma_k$ of length $k\geq 2$, we let $\widehat\sigma=\sigma_2\sigma_1\sigma_3\cdots\sigma_k$ denote the permutation obtained by swapping the first two entries in $\sigma$. We also let $\rev(\sigma)=\sigma_k\cdots\sigma_1$ denote the reverse of $\sigma$. If $\sigma\in S_k$, then we write $\comp(\sigma)=(k+1-\sigma_1)\cdots(k+1-\sigma_k)$ for the complement of $\sigma$. A \emph{descent} (respectively, \emph{ascent}) of $\sigma$ is an index $i$ such that $\sigma_i>\sigma_{i+1}$ (respectively, $\sigma_i<\sigma_{i+1}$). An \emph{ascending run} of $\sigma$ is a maximal consecutive increasing subsequence of $\pi$, while a \emph{descending run} is a maximal consecutive decreasing subsequence. For example, the ascending runs of $7862351$ are $78$, $6$, $235$, and $1$, while the descending runs of $7862351$ are $7$, $862$, $3$, and $51$. We write $|\tau|$ for the length of a sequence $\tau$ of positive integers.  

\subsection{A Preliminary Lemma}
Before we proceed, it will be helpful to record the following lemma, which will allow us to simplify many of our considerations. The proof of this lemma is immediate from the definitions of $s_\sigma$ and $\SC_\sigma$. 

\begin{lemma}\label{LemComplement}
If $\sigma\in S_k$ for some $k\geq 2$, then \[s_{\comp(\sigma)}=\comp\circ s_\sigma\circ\comp\quad\text{and}\quad\SC_{\comp(\sigma)}=\comp\circ\SC_\sigma\circ\comp.\] 
\end{lemma}

\section{When is $\Sort(\SC_\sigma)$ a Permutation Class?}\label{Sec:CharacterizeClasses}

A \emph{permutation class} is a set $\mathcal C$ of permutations that is closed under classical pattern containment. Said differently, this means that if $\sigma,\tau$ are permutations such that $\sigma$ contains $\tau$ classically and $\sigma\in\mathcal C$, then $\tau\in\mathcal C$. Equivalently, a set of permutations is a permutation class if and only if it is of the form $\Av(\tau^{(1)},\tau^{(2)},\ldots)$ for some classical patterns $\tau^{(1)},\tau^{(2)},\ldots$. 

Recall that $\widehat\sigma$ is the permutation obtained by swapping the first two entries in $\sigma$. Furthermore, $\Sort(s_\sigma)=s_\sigma^{-1}(\Av(231))$ is the set of permutations that get sorted into the identity by the map $s\circ s_\sigma$. The following theorem is one of the main results in \cite{Cerbai}. 

\begin{theorem}[\!\!\cite{Cerbai}]\label{ThmCerbaiClass}
Let $\sigma\in S_k$ for some $k\geq 2$. The set $\Sort(s_\sigma)$ is a permutation class if and only if $\sigma=12$ or $\widehat\sigma$ contains the pattern $231$ classically. 
\end{theorem}

Our goal in this section is to prove the following analogue of Theorem~\ref{ThmCerbaiClass} for the consecutive-pattern-avoiding stack-sorting maps $\SC_\sigma$. Recall that $\Sort(\SC_\sigma)=\SC_\sigma^{-1}(\Av(231))$ is the set of permutations that get sorted into the identity by the map $s\circ \SC_\sigma$. 

\begin{theorem}\label{ThmClass}
The set $\Sort(\SC_{12})$ is equal to the permutation class $\Av(213)$, while $\Sort(\SC_{21})$ is not a permutation class. If $\sigma\in S_k$ for some $k\geq 3$, then the set $\Sort(\SC_\sigma)$ is a permutation class if and only if $\sigma$ and $\widehat\sigma$ both contain the pattern $231$ classically. In this case, $\Sort(\SC_\sigma)=\Av(132)$. 
\end{theorem}

We will prove Theorem~\ref{ThmClass} via the following two lemmas. Recall that $\rev(\sigma)$ denotes the reverse of the permutation $\sigma$. 

\begin{lemma}\label{Lem1}
Let $\sigma\in S_k$ for some $k\geq 3$. If $\widehat{\sigma}$ contains $231$ classically, then $\Sort(\SC_\sigma)$ is the set $\Av\left(132,\underline{\rev(\sigma)}\right)$ of permutations that avoid $132$ classically and avoid $\rev(\sigma)$ consecutively.
\end{lemma}

\begin{proof}
In general, it is immediate from the definition of $\SC_\sigma$ that $\SC_\sigma(\pi)=\rev(\pi)$ whenever $\pi$ avoids $\rev(\sigma)$ consecutively. Indeed, in this case, all of the entries in $\pi$ enter the stack and then proceed to exit the stack in the reverse of the order in which they entered. It follows that if $\pi$ avoids $\rev(\sigma)$ consecutively and avoids $132$ classically, then $\SC_\sigma(\pi)=\rev(\pi)\in\Av(231)$. Furthermore, if $\pi$ contains $132$ classically and avoids $\rev(\sigma)$ consecutively, then $\SC_\sigma(\pi)=\rev(\pi)$ contains $231$ classically. This implies that $\pi\not\in\Sort(\SC_\sigma)$.

Suppose now that $\pi$ contains $\rev(\sigma)$ consecutively, and let $i$ be the minimal index such that $\pi_{i}\cdots\pi_{i+k-1}$ is a consecutive occurrence of the pattern $\rev(\sigma)$. Consider sending $\pi$ through a consecutive-$\sigma$-avoiding stack. When $\pi_{i+k-1}$ is about to enter the stack, the top $k-1$ entries in the stack, read from top to bottom, are $\pi_{i+k-2},\ldots,\pi_i$. The entry $\pi_{i+k-2}$ pops out of the stack, and the entry $\pi_{i+k-1}$ enters the stack immediately afterward. Once all of the entries have finally exited the stack, the entries $\pi_{i+k-2},\pi_{i+k-1},\pi_{i+k-3},\ldots,\pi_i$ will appear in this order in the output permutation $\SC_\sigma(\pi)$. However, $\pi_{i+k-2}\pi_{i+k-1}\pi_{i+k-3}\cdots\pi_i$ has the same relative order as $\widehat\sigma$, which in turn contains $231$ classically by hypothesis. This shows that $\SC_\sigma(\pi)\not\in\Av(231)$, so $\pi\not\in\Sort(\SC_\sigma)$. 
\end{proof}

\begin{lemma}\label{Lem2}
If $\widehat\sigma\in\Av_k(231)$ for some $k\geq 3$, then $\Sort(\SC_\sigma)$ is not a permutation class. 
\end{lemma}

\begin{proof}
We first dispose of the case in which $\sigma=231$. A simple computation yields $\SC_{231}(2413)=3142\not\in\Av(231)$ and $\SC_{231}(25314)=54132\in\Av(231)$. This shows that $25314$ is in $\Sort(\SC_{231})$ while $2413$ is not. Since $25314$ classically contains $2413$, it follows that $\Sort(\SC_{231})$ is not a permutation class. We may now assume $\sigma\neq 231$. Because the length of $\sigma$ is at least $3$, we have $\SC_\sigma(132)=231$. This shows that $132\not\in\Sort(\SC_\sigma)$. 

Let us now assume that $\sigma$ classically contains $231$ but is not equal to $231$. One can easily check that $\SC_\sigma(\rev(\sigma))=\widehat\sigma$, and we assumed that $\widehat\sigma\in\Av(231)$. This shows that $\rev(\sigma)\in\Sort(\SC_{\sigma})$. Now note that $\rev(\sigma)$ contains $132$ classically and that $132\not\in\Sort(\SC_\sigma)$. It follows that $\Sort(\SC_{\sigma})$ is not a permutation class in this case.  

Next, assume $\sigma\in\Av(231)$, $\sigma_1<\sigma_2$, and $\sigma_2>2$. These assumptions force $\sigma_1=1$ since, otherwise, the entries $\sigma_1,\sigma_2,1$ would form a classical $231$ pattern in $\sigma$. Let $\pi=\sigma_k'\cdots\sigma_2'\sigma_1'\sigma_2$, where $\sigma_i'=\sigma_i$ if $\sigma_i<\sigma_2$ and $\sigma_i'=\sigma_i+1$ if $\sigma_i\geq\sigma_2$. We will prove that $\pi\in\Sort(\SC_\sigma)$. Since the entries $2,\sigma_2',\sigma_2$ form an occurrence of the classical pattern $132$ in $\pi$ and $132\not\in\Sort(\SC_\sigma)$, this will imply that $\Sort(\SC_\sigma)$ is not a permutation class. One can easily compute that $\SC_\sigma(\pi)=\sigma_2'\sigma_2\sigma_1'\sigma_3'\cdots\sigma_k'$. Notice that the permutation $\sigma_2\sigma_1'\sigma_3'\cdots\sigma_k'$ has the same relative order as $\widehat\sigma$, which avoids $231$ classically by hypothesis. Therefore, if $\SC_\sigma(\pi)$ contains $231$, then the first entry in the occurrence of the $231$ pattern must be $\sigma_2'$. Since $\sigma_2'=\sigma_2+1$, we can replace $\sigma_2'$ with $\sigma_2$ in this subsequence in order to obtain an occurrence of the pattern $231$ in $\SC_\sigma(\pi)$ that starts with the entry $\sigma_2$. However, this contradicts our earlier observation that $\sigma_2\sigma_1'\sigma_3'\cdots\sigma_k'$ avoids $231$. 

For our next case, we assume $\sigma\in\Av(231)$, $\sigma_1=1$, and $\sigma_2=2$. Let $\pi=\sigma_k\cdots\sigma_321(k+1)$ be the concatenation of $\rev(\sigma)$ with the entry $k+1$. Let $\pi'=\sigma_k\cdots\sigma_32(k+1)1(k+2)$. The first $3$ entries of the permutation $\SC_\sigma(\pi)=2(k+1)1\sigma_3\cdots\sigma_k$ form a $231$ pattern, so $\pi\not\in\Sort(\SC_\sigma)$. On the other hand, since $\pi'$ certainly avoids the pattern $\rev(\sigma)$ consecutively, we have $\SC_\sigma(\pi')=\rev(\pi')=(k+2)1(k+1)2\sigma_3\cdots\sigma_k$. The assumption $\sigma\in\Av(231)$ guarantees that $\SC_\sigma(\pi')\in\Av(231)$, so $\pi'\in\Sort(\SC_\sigma)$. The permutation $\pi'$ contains $\pi$ classically, so $\Sort(\SC_\sigma)$ is not a permutation class in this case. 

The final case we need to consider is that in which $\sigma\in\Av(231)$ and $\sigma_1>\sigma_2$. Notice that the hypothesis $\widehat\sigma\in\Av(231)$ forces $\sigma_2=1$ in this case. Let $\pi=\sigma_k'\cdots\sigma_2'\sigma_1'2$, where $\sigma_2'=1$ and $\sigma_i'=\sigma_i+1$ for all $i\neq 2$. We have $\SC_\sigma(\pi)=12\sigma_1'\sigma_3'\cdots\sigma_k'$; note that this permutation avoids $231$ because $\sigma$ avoids $231$. Therefore, $\pi\in\Sort(\SC_\sigma)$. The entries $\sigma_2',\sigma_1',2$ form an occurrence of the pattern $132$ in $\pi$, and we have seen that $132\not\in\Sort(\SC_\sigma)$. Therefore, $\Sort(\SC_\sigma)$ is not a permutation class in this final case. 
\end{proof}

\begin{proof}[Proof of Theorem~\ref{ThmClass}]
It is straightforward to check that if $\pi\in S_n$ for some $n\geq 1$, then the last entry in $\SC_{12}(\pi)$ is $1$. It follows that $\SC_{12}(\pi)$ avoids $231$ if and only if it avoids $12$. Using Lemma~\ref{LemComplement}, we find that \[\Sort(\SC_{12})=\SC_{12}^{-1}(\Av(231))=\SC_{12}^{-1}(\Av(12))=\comp(\SC_{21}^{-1}(\comp(\Av(12))))\] \[=\comp(s^{-1}(\Av(21)))=\comp(\Av(231))=\Av(213).\] On the other hand, the set $\Sort(\SC_{21})$ is not a permutation class because it includes the permutation $35241$ but not the permutation $3241$. 

Now assume $\sigma\in S_k$ for some $k\geq 3$. If $\sigma$ and $\widehat\sigma$ both contain $231$ classically, then it follows from Lemma~\ref{Lem1} that \[\Sort(\SC_\sigma)=\Av\left(132,\underline{\rev(\sigma)}\right)=\Av(132).\] If $\widehat\sigma\in\Av(231)$, then Lemma~\ref{Lem2} tells us that $\Sort(\SC_\sigma)$ is not a permutation class. Now assume $\sigma\in\Av(231)$ and $\widehat\sigma\not\in\Av(231)$. Note that this implies that $\sigma_1>\sigma_2>1$. Since $\SC_\sigma(\rev(\sigma))=\widehat\sigma$, the permutation $\rev(\sigma)$ is not in $\Sort(\SC_\sigma)$. However, the permutation $\pi=(\sigma_k+1)\cdots(\sigma_2+1)1(\sigma_1+1)$, which contains $\rev(\sigma)$ classically, is in $\Sort(\SC_\sigma)$. To see this, note that the condition $\sigma_1>\sigma_2>1$ guarantees that $\pi$ avoids $\rev(\sigma)$ consecutively. Therefore, $\SC_\sigma(\pi)=1(\sigma_1+1)(\sigma_2+1)\cdots(\sigma_k+1)$, and this permutation avoids $231$ because $\sigma$ does. Hence, $\Sort(\SC_\sigma)$ is not a permutation class.  
\end{proof}

\section{Periodic Points}\label{Sec:PeriodicPoints}
In this section, we prove the following theorem, which completely classifies the periodic points of the consecutive-$\sigma$-avoiding stack-sorting map for $\sigma \in S_3$. 

\begin{theorem}\label{thm: periodicmaintheorem}
Let $\sigma\in S_3$. The periodic points of the map $\SC_\sigma:S_n\to S_n$ are precisely the permutations in $\Av_n\left(\underline{\sigma},\underline{\rev(\sigma)}\right)$. When $n\geq 2$, each of these periodic points has period $2$. 
\end{theorem}

Lemma~\ref{LemComplement} tells us that for each permutation pattern $\sigma$, the maps $\SC_\sigma$ and $\SC_{\comp(\sigma)}$ are topologically conjugate (with the discrete topology), with the complementation map $\comp$ serving as the topological conjugacy. It follows that if Theorem~\ref{thm: periodicmaintheorem} holds when $\sigma$ is some specific pattern $\tau$, then it also holds when $\sigma=\comp(\tau)$. Thus, we really only need to prove the theorem when $\sigma\in\{123,132,231\}$. We first handle the case in which $\sigma\in\{132,231\}$.

\begin{proposition}\label{Prop:Periodic132}
For $\sigma \in \{132, 231\}$, the set of periodic points of the map $\SC_{\sigma}:S_n\to S_n$ is the set $\Av_n(\underline{132},\underline{231})$ of permutations in $S_n$ that avoid $132$ and $231$ consecutively (these permutations also avoid $132$ and $231$ classically). When $n\geq 2$, these points have period $2$.
\end{proposition}
\begin{proof}

Note that $\{132,231\}=\{\sigma,\rev(\sigma)\}$. It is clear that if $\pi \in \Av(\underline{132}, \underline{231})$, then $\SC_{\sigma}(\pi) = \rev(\pi) \in \Av(\underline{132}, \underline{231})$, so $\SC_{\sigma}^2(\pi) = \pi$. Hence, $\pi$  is a periodic point of $\SC_{\sigma}$, and the period is $2$ if $n\geq 2$. 

We now prove that every periodic point of $\SC_\sigma:S_n\to S_n$ is in $\Av(\underline{132},\underline{231})$. This is trivial if $n\leq 2$, so we may assume $n\geq 3$ and induct on $n$. Given $\tau\in S_n$, let $\tau^*$ be the permutation in $S_{n-1}$ obtained by deleting the entry $1$ from $\tau$ and then decreasing all remaining entries by $1$. Observe that if the entries $1$ and $2$ appear consecutively in $\tau$, then they also appear consecutively in $\SC_\sigma(\tau)$ and that $(\SC_{\sigma}(\tau))^*=\SC_{\sigma}(\tau^*)$. Now fix $\pi\in S_n$. By induction, we have $\SC_\sigma^t(\pi^*)\in\Av(\underline{132},\underline{231})$ for all sufficiently large $t$.  We will prove that if $t$ is sufficiently large, then the entries $1$ and $2$ appear consecutively in $\SC_\sigma^t(\pi)$. It will then follow that $(\SC_{\sigma}^t(\pi))^*=\SC_\sigma^t(\pi^*)\in \Av(\underline{132},\underline{231})$ for all sufficiently large $t$, which will imply that $\SC_{\sigma}^t(\pi)\in\Av(\underline{132},\underline{231})$ for all sufficiently large $t$. 

Let $D(\tau)$ denote the number of entries between $1$ and $2$ in $\tau$ (regardless of whether $1$ appears before or after $2$ in $\tau$). First, note that neither $1$ nor $2$ will get popped from the stack until after all of the entries of $\tau$ have entered the stack because neither $1$ nor $2$ can trigger the consecutive-$\sigma$-avoiding restriction when at the top of the stack. This immediately implies that $D(\tau) \leq D(\SC_{\sigma}(\tau))$. We further show that if $D(\tau)>0$, then $D(\SC_{\sigma}^2(\tau))<D(\tau)$, implying that $1$ and $2$ are consecutive entries in $\SC_{\sigma}^t(\pi)$ for all sufficiently large $t$. Let $\{1, 2\} = \{\tau_i, \tau_{i+c} \}$ so that $D(\pi) = c-1$, and assume $c\geq 2$. The consecutive subsequence $\tau_i\tau_{i+1} \cdots \tau_{i+c}$ of $\tau$ contains a consecutive occurrence of either $\sigma$ or $\rev(\sigma)$. If it contains a consecutive occurrence of $\rev(\sigma)$, then when $\tau_{i+c}$ enters the stack, the entries between $\tau_i$ and $\tau_{i+c}$ form a strict subset of $\{\tau_{i+1},\ldots, \tau_{i+c-1}\}$. Moreover, these entries are precisely the entries between $1$ and $2$ in $\SC_{\sigma}(\tau)$ since neither $1$ nor $2$ is popped until all entries of $\tau$ have entered the stack. This implies that $D(\SC_{\sigma}(\tau)) < D(\tau)$. If there is no consecutive occurrence of $\rev(\sigma)$ in the consecutive subsequence $\tau_i\tau_{i+1} \cdots \tau_{i+c}$ of $\tau$, then $\tau_{i+c} \tau_{i+c-1} \cdots \tau_i$ is a consecutive subsequence of $\SC_{\sigma}(\tau)$ that contains a consecutive occurrence of $\rev(\sigma)$, so the previous argument guarantees that $D(\SC_{\sigma}^2(\tau)) < D(\SC_\sigma(\tau))\leq D(\tau)$.
\end{proof}

Next, we consider the periodic points of $\SC_{123}$ and $\SC_{321}$. By complementing, it suffices to just consider $\SC_{321}$. To understand this map, it helps to decompose a permutation $\pi$ by writing $\pi = a_1 \cdots a_k$, where the $a_i$'s are the ascending runs (maximal consecutive increasing subsequences). In addition, let $a_i^m$ be the subsequence of $a_i$ consisting of the entries that are not the first or last entries in $a_i$, and let $a_i^{e}$ be the sequence obtained from $a_i$ by deleting $a_i^m$. If $|a_i|$ is $1$ or $2$, then $a_i^{e} = a_i$ and $a_i^{m}$ is an empty sequence. For example, the ascending runs of $4572136$ are $a_1 = 457$, $a_2 = 2$, and $a_3 = 136$, with $a_1^e = 47, a_2^e = 2, a_3^e = 16, a_1^m = 5, a_2^m$ empty, and $a_3^m = 3$.

\begin{lemma} \label{LemSC321 map}
Keeping the notation above, let $\pi = a_1 \cdots a_k$, where the $a_i$'s are ascending runs. Then $\SC_{321}(\pi) = a_1^m \cdots a_k^m \rev(a_1^e \cdots a_k^e) = a_1^m \cdots a_k^m \rev(a_k^e) \cdots \rev(a_1^e)$.
\end{lemma}
\begin{proof}
When the ascending run $a_i$ enters the consecutive-$321$-avoiding stack, each entry in $a^m$ is popped immediately after it enters while $a_i^e$ remains in the stack. After all entries have been pushed into the stack, the entries remaining in the stack are popped out in the reverse of the order in which they entered and thus the output is $ a_1^m \cdots a_k^m \rev(a_1^e \cdots a_k^e) =a_1^m \cdots a_k^m \rev(a_k^e) \cdots \rev(a_1^e)$.
\end{proof}

To identify the periodic points of $\SC_{321}$, we analyze how the number of peaks and valleys of a permutation $\pi$ changes as $\SC_{321}$ is iteratively applied to $\pi$. A \emph{peak} (respectively, \emph{valley}) of $\pi = \pi_1 \cdots \pi_n$ is an index $i\in\{2,\ldots,n-1\}$ such that $\pi_{i-1}<\pi_i>\pi_{i+1}$ (respectively, $\pi_{i-1}>\pi_i<\pi_{i+1}$). If $i$ is a peak (respectively, valley) of $\pi$, then the entry $\pi_i$ is called a \emph{peak top} (respectively, \emph{valley bottom}) of $\pi$. We let $g(\pi)$ denote the total number of indices that are either peaks or valleys of $\pi$. 

\begin{lemma} \label{LemSC321 PV}
Let $\pi = a_1 \cdots a_k\in S_n$, where $a_1,\ldots,a_k$ are the ascending runs of $\pi$. We have $g(\pi) \leq g(\SC_{321}(\pi))$. Moreover, if $|a_k|\geq 2$, then this inequality is an equality if and only if $a_1^m \cdots a_k^m \rev(a_k^e)$ is a decreasing sequence. If $|a_k|=1$, then $g(\pi)= g(\SC_{321}(\pi))$ if and only if $a_1^m \cdots a_k^m \rev(a_k^e)$ is an increasing sequence. 
\end{lemma}
\begin{proof}
We specifically show that the peak tops and valley bottoms of $\pi$ are still peak tops and valley bottoms in $\SC_{321}(\pi)$. The key observation is that each peak top (valley bottom) of $\pi$ is the second (first) entry of $a_j^e$ for some ascending run $a_j$ of length at least $2$, and it is not the first or last entry of $\pi$.

Suppose the entry $\pi_i$ is a peak top of $\pi$. Then $\pi_i$ is the last entry of some ascending run $a_j$ of length at least $2$. We have $a_j^e=\pi_h\pi_i$ for some $\pi_h<\pi_i$. Similarly, $\pi_{i+1}$ is the first entry of the ascending run $a_{j+1}$ and is less than $\pi_i$. By Lemma~\ref{LemSC321 map}, the sequence $\rev(a_{j+1}^e)\rev(a_j^e)$, which itself has $\pi_{i+1}\pi_i\pi_h$ as a consecutive subsequence, is a consecutive subsequence of $\SC_{321}(\pi)$. Consequently, $\pi_i$ is still a peak top of $\SC_{321}(\pi)$.
A completely analogous argument shows that every valley bottom of $\pi$ is also a valley bottom of $\SC_{321}(\pi)$, so $g(\pi)\leq g(\SC_{321}(\pi))$.

For the equality cases, suppose first that $a_k$ has length at least $2$. Note that $g(\pi) = g(\SC_{321}(\pi))$ if and only if no entry in $a_1^m \cdots a_k^m \rev(a_k^e)$ except for the last entry is a peak top or valley bottom in $\SC_{321}(\pi)$. This happens if and only if $a_1^m \cdots a_k^m \rev(a_k^e)$ is an increasing or decreasing sequence. Since $\rev(a_k^e)$ is a decreasing sequence, it follows that $g(\pi) = g(\SC_{321}(\pi))$ if and only if $a_1^m \cdots a_k^m \rev(a_k^e)$ is a decreasing sequence. 

Now suppose $a_k$ has length $1$. We again have $g(\pi) = g(\SC_{321}(\pi))$ if and only if no element in $a_1^m \cdots a_k^m \rev(a_k^e)$ is a peak top or valley bottom in $\SC_{321}(\pi)$. Note that $a_1^m \cdots a_k^m \rev(a_k^e)$ is part of the consecutive subsequence $a_1^m \cdots a_k^m \rev(a_k^e)\rev(a_{k-1}^e)$ in $\SC_{321}(\pi)$. Since $a_k$ has length $1$, it is a single entry that is less than the entry before it in $\pi$; this entry before it is also the entry after it in $a_1^m \cdots a_k^m \rev(a_k^e)\rev(a_{k-1}^e)$. Thus, no entry in $a_1^m \cdots a_k^m \rev(a_k^e)$ is a peak top or valley bottom if and only if $a_1^m \cdots a_k^m \rev(a_k^e)$ is an increasing sequence. 
\end{proof}

By the previous lemma, it is clear that if $g(\pi) < g(\SC_{321}(\pi))$, then $\pi$ cannot be a periodic point. We make use of this fact when proving the next few results. 

\begin{corollary} \label{Lem inc3}
Let $\pi = a_1 \cdots a_k$, where the $a_i$'s denote ascending runs. If $|a_k|\geq 3$ or $|a_1|=|a_2|=1$, then $\pi$ is not a periodic point of $\SC_{321}$.
\end{corollary}
\begin{proof}
If $|a_k|\geq 3$, then $a_k^m\rev(a_k^e)$ is not a decreasing sequence, so $g(\pi) < g(\SC_{321}(\pi))$ by Lemma~\ref{LemSC321 PV}. Hence, $\pi$ is not a periodic point of $\SC_{321}$. If $|a_1|=|a_2|=1$, then $\SC_{321}(\pi)$ is a permutation ending in an ascending run of length at least $3$, so it follows from the preceding argument that $\SC_{321}(\pi)$ is not a periodic point of $\SC_{321}$. Hence, $\pi$ is not a periodic point in this case either.  
\end{proof}

\begin{corollary} \label{Lem dec3}
Let $\pi = a_1 \cdots a_k\in S_n$ for some $n\geq 3$, where the $a_i$'s denote ascending runs. If $|a_1|\geq 3$ or $|a_{k-1}|=|a_{k}|=1$, then $\pi$ is not a periodic point of $\SC_{321}$.
\end{corollary}
\begin{proof}

Suppose $\pi$ is a periodic point of $\SC_{321}$ and $|a_1|\geq 2$. Let $\ell$ be the period of $\pi$ so that $\SC_{321}^\ell(\pi)=\pi$. Let $\pi_1$ denote the first entry of $\pi$. By Lemma~\ref{LemSC321 PV}, we must have $g(\pi)=g(\SC_{321}^t(\pi))$ for all $t\geq 1$. By Lemma~\ref{LemSC321 map}, the last two entries of $\SC_{321}(\pi)$ form the sequence $\rev(a_1^e)$; in particular, the last ascending run of $\SC_{321}(\pi)$ has length $1$ and is simply $\pi_1$. Since $\SC_{321}(\pi)$ ends in an ascending run of length $1$ and $g(\pi)=g(\SC_{321}(\pi))$, Lemma~\ref{LemSC321 PV} guarantees that $\SC_{321}^2(\pi)$ begins with an ascending run that contains $\pi_1$, so $\SC_{321}^2(\pi)_1 \leq \pi_1$. In fact, the first ascending run of $\SC_{321}^2(\pi)$ must actually contain $a_1^e$, so $\SC_{321}^2(\pi)$ is a periodic point of $\SC_{321}$ that starts with an ascending run of length at least $2$. Moreover, if $\SC_{321}(\pi)$ has an ascending run of length at least $3$, then the first ascending run of $\SC_{321}^2(\pi)$ actually begins with an entry that is smaller than $\pi_1$, implying that $\SC_{321}^2(\pi)_1 < \pi_1$. Repeating this argument shows that $(\SC_{321}^{2t}(\pi)_1)_{t\geq 1}$ is a weakly decreasing sequence and is nonconstant if $\SC_{321}^{2t-1}(\pi)$ contains an ascending run of length at least $3$ for some $t \geq 1$. Since $\SC_{321}^{k\ell}(\pi)=\pi$ for all integers $k\geq 1$, it follows that for every $t\geq 1$, the permutation $\SC_{321}^{2t-1}(\pi)$ does not have an ascending run of length at least $3$. This implies that $\SC_{321}^{2t}(\pi) = \rev(\SC_{321}^{2t-1}(\pi))$ for all $t \geq 1$. 

Now consider $\pi_2$, the second entry of $\pi$. As previously stated, $\SC_{321}(\pi)$ ends in $\rev(a_1^e)$, so the second-to-last entry of $\SC_{321}(\pi)$ is greater than or equal to $\pi_2$. This implies that the second entry of $\SC_{321}^2(\pi)$ is greater than or equal to $\pi_2$, as $\SC_{321}^2(\pi) = \rev(\SC_{321}(\pi))$. Repeating this argument shows that the sequence $(\SC_{321}^{2t}(\pi)_2)_{t\geq 1}$ is weakly increasing. As $\SC_{321}^{2k\ell}(\pi)=\pi$ for all $k\geq 1$, this sequence must actually be constant. If $|a_1|\geq 3$, then the second-to-last entry of $\SC_{321}(\pi)$, which is also the second entry of $\SC_{321}^2(\pi)$, is actually strictly greater than $\pi_2$. This is a contradiction. 

Now suppose $\pi$ is a periodic point and $|a_{k-1}|=|a_k|=1$. This means we must have $g(\pi)=g(\SC_{321}(\pi))$, so it follows from Lemma~\ref{LemSC321 PV} that $a_1^m\cdots a_k^m\rev(a_k^e)$ is an increasing sequence. Notice also that the first $3$ entries of the sequence $\rev(a_k^e)\rev(a_{k-1}^e)\rev(a_{k-2}^e)$ appear in increasing order (we must have $k\geq 3$ since $n\geq 3$ and $|a_{k-1}|=|a_k|=1$). This shows that $\SC_{321}(\pi)$ is a periodic point beginning with an ascending run of length at least $3$, which contradicts what we found in the previous two paragraphs. 
\end{proof}

The following proposition proves Theorem~\ref{thm: periodicmaintheorem} for $\sigma=321$, thereby completing the proof of the entire theorem. Indeed, we already proved the theorem for $\sigma\in\{132,213,231,312\}$, and we can use Lemma~\ref{LemComplement} to see that the next proposition implies Theorem~\ref{thm: periodicmaintheorem} for $\sigma=123$ as well. 

\begin{proposition}
The set of periodic points of $\SC_{321}:S_n\to S_n$ is $\Av_n(\underline{123}, \underline{321})$, the set of permutations in $S_n$ that avoid $123$ and $321$ consecutively. When $n\geq 2$, these points have period $2$.
\end{proposition}
\begin{proof}
It is clear that if $\pi \in \Av(\underline{123}, \underline{321})$, then $\SC_{321}(\pi) = \rev(\pi) \in \Av(\underline{123}, \underline{321})$, so $\SC_{321}^2(\pi) = \pi$. Hence, $\pi$  is a periodic point of $\SC_{321}$, and the period is $2$ if $n\geq 2$. 

Suppose that $\pi \notin \Av(\underline{123}, \underline{321})$; we will show that $\pi$ is not a periodic point of $\SC_{321}$. If $\pi$ does not contain a consecutive $321$ pattern but does contain a consecutive $123$ pattern, then $\SC_{321}(\pi) = \rev(\pi)$ contains a consecutive $321$, so we may simply assume that $\pi \notin \Av(\underline{321})$ (otherwise, the same analysis applies to $\SC_{321}(\pi)$). Suppose instead that $\pi$ is a periodic point, and let $a_1,\ldots,a_k$ be its ascending runs. Note that we have $g(\pi) = g(\SC_{321}(\pi))$. Furthermore, the sequence $a_1^m\cdots a_k^m$ is nonempty because $\pi$ contains a consecutive $321$ pattern. It then follows from Lemma~\ref{LemSC321 map} that $\SC_{321}(\pi)$ is a periodic point of $\SC_{321}$ that begins with the sequence $a_1^m \cdots a_k^m \rev(a_k^e)$. If $|a_k|\geq 2$, then $a_1^m \cdots a_k^m \rev(a_k^e)$ has length at least $3$, and Lemma~\ref{LemSC321 PV} tells us that this sequence is decreasing.  This implies that the periodic point $\SC_{321}(\pi)$ starts with a decreasing sequence of length at least $3$, contradicting Corollary~\ref{Lem inc3}. Hence, we must have $|a_k|=1$. Consider $\pi_{n-1}$, which is last entry in $a_{k-1}$. We must have $\pi_{n-1}>\pi_n$, so it follows from Lemma~\ref{LemSC321 PV} that $a_1^m \cdots a_k^m \rev(a_k^e)\pi_{n-1}$ is an increasing sequence of length at least $3$. This shows that $\SC_{321}(\pi)$ is a periodic point whose first ascending run has length at least $3$, which contradicts Corollary~\ref{Lem dec3}. 
\end{proof}
\section{The Maps $\SC_{132}$ and $\SC_{312}$}\label{Sec:132and312}

In this section, we explore the maps $\SC_{132}$ and $\SC_{312}$ from both a sorting point of view and a dynamical point of view. We do not have much to say about $\SC_{312}$ from a sorting point of view because we were not able to enumerate the set $\Sort(\SC_{312})$. The initial values of the sequence $(|\Sort_n(\SC_{312})|)_{n\geq 0}$ are \[1, 1, 2, 5, 15, 50, 179, 675, 2649, 10734;\] this sequence appears to be new.
However, we \emph{will} be able to characterize and enumerate the permutations in $\Sort(\SC_{132})$. Specifically, we show that they are in bijection with Dyck paths and are thus enumerated by the Catalan numbers. 

From a dynamical point of view, we will be interested in the number of preimages of permutations under $\SC_{132}$ and $\SC_{312}$. In this setup, it suffices to focus our attention on $\SC_{132}$. Indeed, since $312=\comp(132)$, Lemma~\ref{LemComplement} implies that $|\SC_{312}^{-1}(\pi)|=|\SC_{132}^{-1}(\comp(\pi))|$ for all $\pi\in S_n$. Consequently, we will not need to discuss the map $\SC_{312}$ directly in this section. 

Let us begin with the set $\Sort(\SC_{132})$. Recall that a \emph{left-to-right minimum} of a permutation $\pi=\pi_1\cdots\pi_n$ is an entry $\pi_i$ such that $\pi_i<\pi_j$ for all $1\leq j<i$. For example, the left-to-right minima of $4572163$ are $4$, $2$, and $1$. 

\begin{proposition}\label{lm: 132 decon}
Let $\pi = a_1 \cdots a_k\in S_n$, where the subsequences $a_i$ are the ascending runs of $\pi$. For each $i$, let us write $a_i = m_i t_i$, where $m_i$ is the first entry of $a_i$ and $t_i$ is the (possibly empty) string obtained by removing $m_i$ from $a_i$. Then $\SC_{132}(\pi) \in \Av(231)$ if and only if each $m_i$ is a left-to-right minimum of $\pi$ and the concatenation $\rev(t_1)\cdots \rev(t_k)$ is a decreasing sequence. 
\end{proposition}
\begin{proof}
Let us first assume $\SC_{132}(\pi)\in\Av(231)$. We will prove by induction that the entries $m_1,\ldots,m_k$ are all left-to-right minima of $\pi$. Since $m_1$ is the first entry of $\pi$, it is trivially a left-to-right minimum. Now consider $1\leq i<k$, and suppose that $m_1,\ldots,m_i$ are all left-to-right minima. Suppose by way of contradiction that $m_{i+1}$ is not a left-to-right minimum. This implies that if we write $t_i = b_1 \cdots b_j$, where $j=|t_i|$, then $j\geq 1$ and $m_i<m_{i+1}<b_j$. Let $\ell$ be the smallest index such that $b_\ell > m_{i+1}$. When we send $\pi$ through the consecutive-$132$-avoiding stack, the consecutive subsequence $a_i$ will enter the stack consecutively, so $\rev(a_i)=b_j\cdots b_1m_i$ will be at the top of the stack when $m_{i+1}$ is next in line to be pushed into the stack. The entries $b_j,\ldots,b_{\ell+1}$ will be popped, and then $m_{i+1}$ will be pushed into the stack on top of $b_\ell$. Since $b_\ell>m_{i+1}$, the entry $m_{i+1}$ will not get popped out of the stack until after all of the entries of $\pi$ have entered the stack. It follows that $m_{i+1},b_\ell,m_i$ form an occurrence of the classical pattern $231$ in $\SC_{132}(\pi)$; this is our desired contradiction. It is straightforward to see that $\SC_{132}(a_1 \cdots a_k) = \rev(t_1) \cdots \rev(t_k) m_k \cdots m_1$, where $m_k = 1$. Since we are assuming $\SC_{132}(a_1\cdots a_k)$ avoids $231$, the sequence $\rev(t_1) \cdots \rev(t_k)$ must be decreasing. 

For the reverse direction, suppose that the entries $m_1,\ldots,m_k$ are left-to-right minima and that the sequence $\rev(t_1) \cdots \rev(t_k)$ is decreasing. As above, $\SC_{132}(\pi) = \rev(t_1) \cdots \rev(t_k)m_k \cdots m_1$, which is a decreasing sequence followed by an increasing sequence. Such a permutation necessarily avoids $231$.
\end{proof}

\begin{remark}
The proof of Proposition~\ref{lm: 132 decon} yields the somewhat-unexpected fact that if a permutation is in the image of $\SC_{132}$ and avoids the pattern $231$, then it also avoids $132$. 
\end{remark}

 Proposition~\ref{lm: 132 decon} tells us that each $\pi\in\Sort(\SC_{132})$ is determined by the first entries $m_i$ of its ascending runs as well as the lengths $|a_1|,\ldots,|a_k|$ of its ascending runs. Furthermore, these quantities must satisfy $n+1 - m_i \geq |a_1| + \cdots + |a_i|$ for all $i$. Indeed, the entries in $t_1$ must be the $|a_1|-1$ largest elements of $[n]\setminus\{m_1,\ldots,m_k\}$, the entries in $t_2$ must be the next $|a_2|-1$ largest elements, and so on. The condition that  $n+1 - m_i \geq |a_1| + \cdots + |a_i|$ for all $i$ guarantees that the $m_i$'s are left-to-right minima. 

A \emph{Dyck path of semilength} $n$ is a word over the alphabet $\{U,D\}$ that contains $n$ copies of each letter and has the additional property that each of its prefixes contains at least as many occurrences of $U$ as occurrences of $D$. Let ${\bf D}_n$ denote the set of Dyck paths of semilength $n$. It is well known that $|{\bf D}_n|=C_n=\frac{1}{n+1}{2n\choose n}$ is the $n$th Catalan number. Given $\pi\in S_n$ such that $\SC_{132}(\pi) \in \Av(231)$, write $\pi = a_1 \cdots a_k$ and $a_i = m_i t_i$ as in Proposition~\ref{lm: 132 decon}. Let $\Phi(\pi)$ be the Dyck path $U^{m_0-m_1}D^{|a_1|}U^{m_1-m_2}D^{|a_2|}\cdots U^{m_{k-1}-m_k}D^{|a_k|}$, where we make the convention $m_0=n+1$. 

\begin{example}
Let $\pi=589436712$, and note that $\SC_{132}(\pi)=987621345\in\Av(231)$. We have $a_1=589$, $a_2=4$, $a_3=367$, and $a_4=12$. The first entries of these ascending runs are $m_1=5$, $m_2=4$, $m_3=3$, and $m_4=1$. Therefore, $\Phi(\pi)=UUUUUDDDUDUDDDUUDD$.
\end{example}

\begin{theorem} \label{Thm 132bij 1}
The map $\Phi:\Sort_n(\SC_{132})\to{\bf D}_n$ is a bijection. Consequently, \[|\Sort_n(\SC_{132})|=C_n.\] 
\end{theorem}

\begin{proof}
Preserve the notation from above. The total number of occurrences of the letter $U$ in $\Phi(\pi)$ is $(m_0-m_1)+\cdots+(m_{k-1}-m_k)=m_0-m_k=(n+1)-1=n$, and the total number of occurrences of $D$ is $|a_1|+\cdots+|a_k|=n$. As mentioned above, we have $n+1-m_i\geq|a_1|+\cdots+|a_i|$ for all $i$; this guarantees that $\Phi(\pi)$ is indeed a Dyck path of semilength $n$. Moreover, by the previous discussion, $\Phi$ is injective. 

To prove surjectivity, suppose we are given a Dyck path $\Lambda=U^{\gamma_1}D^{\delta_1}U^{\gamma_2}D^{\delta_2}\cdots U^{\gamma_k}D^{\delta_k}$ for some positive integers $\gamma_1,\ldots,\gamma_k,\delta_1,\ldots,\delta_k$. Define $m_i=n+1-\gamma_1-\cdots-\gamma_i$. Using the fact that $\Lambda$ is a Dyck path, one can verify that there exists a permutation $\pi\in S_n$ with ascending runs $a_1,\ldots,a_k$ satisfying the conditions in Proposition~\ref{lm: 132 decon} with $a_i=m_it_i$ and $|a_i|=\delta_i$ for all $i$. This permutation is in $\Sort_n(\SC_{132})$ and satisfies $\Phi(\pi)=\Lambda$.  
\end{proof}

The next series of results determines the maximum number of preimages of permutations under $\SC_{132}$ and $\SC_{312}$. As mentioned above, it is only necessary to consider $\SC_{132}$.

\begin{lemma} \label{lm: swap}
Let $\pi\in S_n$. Suppose there exists an entry $i\in[n-1]$ that occurs before $i+1$ but not immediately before $i+1$ in $\pi$. Let $\pi'$ be the permutation obtained by swapping $i$ and $i+1$ in $\pi$. Then,
\begin{equation*}
    |\SC_{132}^{-1}(\pi)| \leq |\SC_{132}^{-1}(\pi')|.
\end{equation*}
\end{lemma}
\begin{proof}
Suppose $\SC_{132}(\tau) = \pi$, and let $\tau'$ be the permutation obtained by swapping $i$ and $i+1$ in $\tau$. We will prove that $\SC_{132}(\tau') = \pi'$. To do so, it suffices to show that the sequence of ``push'' and ``pop'' operations that is used to send $\tau$ through the consecutive-$132$-avoiding stack is exactly the same as the sequence used to send $\tau'$ through the stack. 

When sending $\tau$ through the consecutive-$132$-avoiding stack, the entries $i$ and $i+1$ are never consecutive entries in the stack. Indeed, if they were, then $i$ would sit on top of $i+1$ because $i$ appears before $i+1$ in the output permutation $\pi$. However, if $i$ sits on top of $i+1$, then these two entries will appear consecutively in the output $\pi$, contrary to our hypothesis. Now consider the sequences of ``push'' and ``pop'' operations that send $\tau$ and $\tau'$ through the stack. Note that in each sequence, the first two operations are necessarily pushes. Now suppose the first two sequences agree in their first $k$ operations. Suppose further that after the first $k$ operations, the stack-sorting procedure that is being applied to $\tau$ (respectively, $\tau'$) has $b$ (respectively, $b'$) as the top entry in the stack, has $c$ (respectively, $c'$) as the second-to-top entry in the stack, and has $a$ (respectively, $a'$) as the entry next in line to enter the stack. We show that the next operation is the same as well. 

If none of $a$, $b$, and $c$ are $i$ or $i+1$, then it must be that $a = a'$, $b = b'$, and $c = c'$, so the claim is trivially true. Next, if exactly one of $a$, $b$, and $c$ is $i$ or $i+1$, then $abc$ and $a'b'c'$ have the same relative order, so the claim is true is in this case as well. Finally, suppose exactly two of $a,b,c$ are $i$ or $i+1$. Because $i$ and $i+1$ cannot appear consecutively in the stack, we cannot have $\{b,c\}=\{i,i+1\}$. If $\{a,b\}=\{i,i+1\}$, then one of the two sorting procedures (sorting $\tau$ or $\tau'$) would have $i$ and $i+1$ appear consecutively in the stack, which is again impossible. Therefore, it must be that $(a, c) = (i, i+1)$ or $(a',c') = (i,i+1)$. The two cases are symmetric, so we may assume the former. Note that $abc$ cannot have the same relative order as $132$ as, otherwise, $i$ and $i+1$ would end up being consecutive in the stack once $b$ gets popped out and $c=i+1$ gets pushed in. Thus, the next operation for the $\tau$ case is to push $a$ into the stack. Likewise $a' = i+1$ and $c' = i$, so $a'b'c'$ cannot have the same relative order as $132$. The next operation in the $\tau'$ case is therefore to push $a'$ into the stack. By induction, the sequences of operations when sending $\tau$ and $\tau'$ through the consecutive-$132$-avoiding stack are identical.

Since the map $\tau \mapsto \tau'$ is one-to-one, this shows that $|\SC_{132}^{-1}(\pi)| \leq |\SC_{132}^{-1}(\pi')|$, as desired.
\end{proof}

\begin{corollary} \label{lm: aug1}
For every positive integer $n$, we have 
\[\max_{\pi\in\Av_n(132,213)}|\SC_{132}^{-1}(\pi)| = \max_{\pi\in S_n}|\SC_{132}^{-1}(\pi)|. \]
\end{corollary}
\begin{proof}
Suppose $\pi\in S_n$. If $\pi$ classically contains an occurrence of either $132$ or $213$, then there exists some $i\in[n-1]$ such that $i$ and $i+1$ are not consecutive in $\pi$ and $i$ occurs before $i+1$. Lemma~\ref{lm: swap} tells us that swapping these entries yields a permutation with at least as many preimages under $\SC_{132}$. By repeatedly performing such swaps, we must eventually reach a permutation in $\Av_n(132,213)$ (the sequence of swaps must terminate because each swap decreases the number of inversions in the permutation by $1$), and this permutation has at least as many preimages under $\SC_{132}$ as $\pi$.
\end{proof}

A permutation in $S_n$ is called \emph{reverse-layered} if the set of entries in each of its ascending runs forms an interval of consecutive integers. For example, $5674231$ is reverse-layered because the sets of entries in its ascending runs are $\{5,6,7\}$, $\{4\}$, $\{2,3\}$, and $\{1\}$. It is well known \cite{Linton} that $\Av_n(132,213)$ is precisely the set of reverse-layered permutations in $S_n$. Note that each reverse-layered permutation is uniquely determined by its set of descents, which we can encode via a word of length $n$ over the alphabet $\{A,D\}$ that starts with $A$. More precisely, the reverse-layered permutation $\pi=\pi_1\cdots\pi_n$ corresponds to the word $b_1\cdots b_n$, where $b_1=A$ and for $2\leq j\leq n$, we have $b_j=A$ if $\pi_{j-1}<\pi_j$ and $b_j=D$ if $\pi_{j-1}>\pi_j$. Let $d_\pi(k)$ denote the number of occurrences of $D$ in the subword $b_1\cdots b_k$. Let $a_\pi(k)$ denote the number of occurrences of $A$ in the word $b_{k+1}\cdots b_n$ that appear to the right of an occurrence of the letter $D$ in $b_{k+1}\cdots b_n$. For example, the reverse-layered permutation $\pi=78634512$ corresponds to the word $AADDAADA$. We have \[(d_\pi(0),\ldots,d_\pi(8))=(0,0,0,1,2,2,2,3,3)\quad\text{and}\quad(a_\pi(0),\ldots,a_\pi(8))=(3,3,3,3,1,1,1,0,0).\]  Using this notation, we can obtain a nice formula for $|\SC_{132}^{-1}(\pi)|$ when $\pi$ is reverse-layered. The reader may find it helpful to refer to Example~\ref{Example1} while reading the following proof.

\begin{theorem} \label{th: layer preim}
Let $\pi\in \Av_n(132,213)$ be a reverse-layered permutation in $S_n$, and preserve the definitions of $d_\pi(k)$ and $a_\pi(k)$ from above. We have 
\begin{equation*}
    |\SC_{132}^{-1}(\pi)| = \sum_{k = 0}^{n} {d_{\pi}(k) + a_{\pi}(k) \choose k}.
\end{equation*}
\end{theorem}

\begin{proof}
    Let $\SC_{132,k}^{-1}(\pi)$ denote the set of permutations $\tau\in\SC_{132}^{-1}(\pi)$ such that when $\tau$ is sent through the consecutive-$132$-avoiding stack, exactly $k$ entries get popped out of the stack before all entries have entered the stack. We will show that \[|\SC_{132,k}^{-1}(\pi)|={d_{\pi}(k) + a_{\pi}(k) \choose k},\] which will complete the proof. Note that $|\SC_{132,0}^{-1}(\pi)|=1$ because the only element of $\SC_{132,0}^{-1}(\pi)$ is $\rev(\pi)$. Indeed, $\SC_{132}$ does actually send $\rev(\pi)$ to $\pi$ because $\pi$ is reverse-layered (hence, $\rev(\pi)$ avoids $231$). Therefore, we may assume that $k\geq 1$ in what follows. 
    
    Suppose $\tau\in\SC_{132,k}^{-1}(\pi)$. When we send $\tau$ through the consecutive-$132$-avoiding stack, the $k$ entries popped before all entries have entered the stack must be exactly $\pi_1, \ldots, \pi_k$, and they must be popped in that order. The other entries $\pi_{k+1},\ldots,\pi_n$ get popped out of the stack later in this order, so they must appear in $\tau$ in the order $\pi_n,\ldots,\pi_{k+1}$. Thus, when constructing such a preimage $\tau$, we first insert each entry $\pi_i$ with $1\leq i\leq k$ between a pair of entries $\pi_{j(i)},\pi_{j(i)+1}$ with $k+1\leq j(i)\leq n$ and then reverse the entire permutation. When we do so, the following conditions must be satisfied:
   
    \begin{itemize}
        \item the string $\pi_{j(i)}\pi_i\pi_{j(i)+1}$ has the same relative order as $132$;
        \item $j(1)\geq\cdots\geq j(k)$;
        \item if $j(i)=j(i+1)$, then $\pi_i>\pi_{i+1}$ and $\pi_{i+1}\pi_i$ is a consecutive subsequence of $\tau$.
    \end{itemize}
    The first and third conditions ensure that the entries $\pi_1,\ldots,\pi_k$ actually get popped before all entries have entered the stack; the second and third conditions ensure that they are popped in the correct order.
    Moreover, each placement of the entries $\pi_1, \dots, \pi_k$ between consecutive entries in $\pi_{k+1}\cdots\pi_n$ such that the three conditions are satisfied yields a unique preimage $\tau\in\SC_{132,k}^{-1}(\pi)$. 
    
    To determine the number of valid placements, let $\pi_\ell$ be the last entry in the ascending run of $\pi$ that includes $\pi_k$. When we insert the entry $\pi_k$ between the entries $\pi_{j(k)},\pi_{j(k)+1}$, the first condition above guarantees that $j(k)\geq\ell+1$, so it follows from the second condition that $j(1)\geq\cdots\geq j(k)\geq \ell+1$. Thus, for each $i\in[k]$, the possible choices for the index $j(i)$ are precisely the ascents counted by $a_{\pi}(k)$.
    To satisfy the third condition, note that we can have $j(i)=j(i+1)$ only if $i$ is a descent of $\pi$. In summary, this proves that the number of ways to choose a preimages $\tau$ is equal to the number of ways to choose the indices $j(1),\ldots,j(k)$ from among the ascents counted by $a_\pi(k)$ such that $j(1)\geq\cdots\geq j(k)$, where the inequality $j(i)\geq j(i+1)$ is actually strict whenever $i$ is an ascent of $\pi$.  
    If we decide that there are precisely $t$ descents $i$ such that $j(i)=j(i+1)$, then there are $\binom{d_\pi(k)}{t}$ choices for these descents. There are then $\binom{k}{k-t}$ choices for the set $\{j(1),\ldots,j(k)\}$ of values taken by the indices $j(1),\ldots,j(k)$. These choices uniquely determine $j(1),\ldots,j(k)$.
    Together, this shows that
    \begin{equation*}
        |\SC_{132,k}^{-1}(\pi)|=\sum_{t=0}^{d_{\pi}(k)} {d_\pi(k)\choose t}{a_{\pi}(k) \choose k - t} = {d_{\pi}(k) + a_{\pi}(k) \choose k},
    \end{equation*}
    where the second equality is due to Vandermonde's Identity. 
\end{proof}

\begin{example}\label{Example1}
Let us take $\pi$ to be the permutation $12\,13\,11\,8\,9\,10\,6\,7\,5\,1\,2\,3\,4$. Suppose we want to construct a preimage $\tau\in\SC_{132,5}^{-1}(\pi)$. The ascents counted by $a_\pi(5)$ are $7,10,11,12$, which correspond to the following blanks: $6\underline{\hspace{.2cm}}7\,5\,1\underline{\hspace{.2cm}}2\underline{\hspace{.2cm}}3\underline{\hspace{.2cm}}4$. Therefore, $a_\pi(5)=4$. We need to insert the entries $12,13,11,8,9$ into some of these blanks. As an example illustrating the notation from the proof of Theorem~\ref{th: layer preim}, placing the entry $\pi_2=13$ in the blank between the entries $\pi_{10}=1$ and $\pi_{11}=2$ would correspond to setting $j(2)=10$. The indices in $\{1,\ldots,5\}$ that are descents of $\pi$ are $2$ and $3$ (meaning $d_\pi(5)=2$), so the proof of the theorem tells us that we must choose $j(1),\ldots,j(5)\in\{7,10,11,12\}$ such that $j(1)>j(2)\geq j(3)\geq j(4)>j(5)$. The number of ways to make this choice is \[\sum_{t=0}^2{2\choose t}{4\choose 5-t}={2+4\choose 5}.\] 
For one specific choice, we can take $j(1)=12$, $j(2)=j(3)=j(4)=10$, and $j(5)=7$. This corresponds to inserting the entries into the blanks to produce the permutation $6\,9\,7\,5\,1\,13\,11\,8\,2\,3\,12\,4$. Reversing this permutation produces $\tau=4\,12\,3\,2\,8\,11\,13\,1\,5\,7\,9\,6$, which is indeed in $\SC_{132,5}^{-1}(\pi)$. 
\end{example}

\begin{lemma} \label{lm: DtoA}
Let $\pi = \pi_1 \cdots \pi_n\in S_n\setminus\{123\cdots n\}$, and let $b = b_1 \cdots b_n$ be its corresponding word over the alphabet $\{A,D\}$ as above. Let $i$ be the smallest positive integer such that $b_{i+1}=A$, and let $a = a_{\pi}(i-1)$. Finally, let $\pi'$ be the permutation whose corresponding word over $\{A,D\}$ is the following modification of $b$: 
\[b'=\begin{cases} \underbrace{AD\cdots DD}_{i+1}b_{i+2} \cdots b_n & \mbox{if } a \geq i; \\  \underbrace{AD\cdots D}_ib_{i+2} \cdots b_nA & \mbox{if } a < i. \end{cases}\]
 Then $|\SC_{132}^{-1}(\pi)| \leq |\SC_{132}^{-1}(\pi')|$.
\end{lemma}

\begin{proof}
Suppose $a \geq i$. Then $b = \underbrace{AD\cdots DA}_{i+1}b_{i+2} \cdots b_n$ and $b' = \underbrace{AD\cdots DD}_{i+1}b_{i+2} \cdots b_n$, with permutations $\pi$ and $\pi'$ corresponding to $b$ and $b'$ respectively. Note that $a = a_{\pi}(i-1) = \cdots = a_{\pi}(1)$. Furthermore, $d_\pi(r)=r-1$ for all $1\leq r\leq i$. By Theorem \ref{th: layer preim} and the Hockey Stick Identity, 
\begin{align*}
    |\SC_{132}^{-1}(\pi)| &= 1 + {a \choose 1} + \cdots + {a+i-2 \choose i-1} +  \sum_{r = i}^{n} {d_{\pi}(r) + a_{\pi}(r) \choose r} \\
    &= 1 + {a+i-1 \choose i-1} +  \sum_{r = i}^{n} {d_{\pi}(r) + a_{\pi}(r) \choose r}.
\end{align*}
It is straightforward to see that $d_{\pi}(i) = d_{\pi}(i+1)$ and $a_{\pi}(i) = a_{\pi}(i+1)$, so by Pascal's Identity,
\begin{equation*}
    |\SC_{132}^{-1}(\pi)| = 1 + {a+i-1 \choose i-1}  + {d_{\pi}(i+1) + a_{\pi}(i+1) + 1 \choose i+1}+ \sum_{r = i+2}^{n} {d_{\pi}(r) + a_{\pi}(r) \choose r}.
\end{equation*}

As for the preimages of $\pi'$, notice that for $r \geq i + 1$, we have $d_{\pi'}(r) = d_{\pi}(r) + 1$ and $a_{\pi'}(r) = a_{\pi}(r)$. Furthermore, we have $d_{\pi'}(r)=d_\pi(r)=r-1$ and $a_{\pi'}(r)\geq a_\pi(r)-1$ for all $1\leq r\leq i$. Thus, 
\begin{align*}
    |\SC_{132}^{-1}(\pi')| \geq 1 + &{a-1 \choose 1} + \cdots + {a+i-3 \choose i-1} +  {a+i-2 \choose i} \\
    &+{d_{\pi}(i+1) + a_{\pi}(i+1) + 1 \choose i+1} +   \sum_{r = i+2}^{n} {d_{\pi}(r) + a_{\pi}(r) \choose r} \\
    = 1 + &{a+i-1 \choose i} +  {d_{\pi}(i+1) + a_{\pi}(i+1) + 1 \choose i+1} +   \sum_{r = i+2}^{n} {d_{\pi}(r) + a_{\pi}(r) \choose r}. 
\end{align*}
This proves that if $a = a_{\pi}(i-1) \geq i$, then $|\SC_{132}^{-1}(\pi)| \leq |\SC_{132}^{-1}(\pi')|$ as desired.

Now suppose $a < i$ so that $b =  \underbrace{AD\cdots DA}_{i+1}b_{i+2} \cdots b_n$ and $b' = \underbrace{AD\cdots D}_{i}b_{i+2} \cdots b_nA$, with permutations $\pi$ and $\pi'$ corresponding to $b$ and $b'$ respectively. We may suppose that $b_j = D$ for some $j \geq i+2$ as, otherwise, $b = b'$ and the inequality is trivially true. Thus, $a_{\pi'}(i) = a_{\pi}(i) + 1$ and  $d_{\pi'}(i) = d_{\pi}(i)$. It is straightforward to see that $d_{\pi'}(r) = d_{\pi}(r)$ and $a_{\pi'}(r) = a_{\pi}(r)$ for $1 \leq r \leq i-1$. Furthermore, $d_{\pi}(r) = d_{\pi'}(r-1)$ and $a_{\pi}(r) \leq a_{\pi'}(r-1)$ for $r \geq i+2$. Observe also that $\binom{d_{\pi'}(n)+a_{\pi'}(n)}{n}=\binom{d_{\pi'}(n)}{n}=0$. It follows that,
\begin{align*}
    |\SC_{132}^{-1}(\pi')| - |\SC_{132}^{-1}(\pi)| &= {d_{\pi}(i) + a_{\pi}(i) + 1 \choose i} - {d_{\pi}(i) + a_{\pi}(i) \choose i} - {d_{\pi}(i+1) + a_{\pi}(i+1) \choose i+1}  \\
    &\hspace{.5cm}+\sum_{r = i+1}^{n-1} {d_{\pi'}(r) + a_{\pi'}(r)  \choose r} - \sum_{r=i+2}^n{d_{\pi}(r) + a_{\pi}(r) \choose r} \\
    &\geq {d_{\pi}(i) + a_{\pi}(i) + 1 \choose i} - {d_{\pi}(i) + a_{\pi}(i) \choose i} - {d_{\pi}(i+1) + a_{\pi}(i+1) \choose i+1}  \\
    &\hspace{.5cm}+\sum_{r = i+2}^{n} \left[{d_{\pi}(r) + a_{\pi}(r)  \choose r-1} - {d_{\pi}(r) + a_{\pi}(r) \choose r}\right].
\end{align*}

Since $a < i$ by assumption, $a_{\pi}(r) < a < i < r$ for all $r \geq i+2$. Moreover, $d_{\pi}(r) < r$ trivially, so it follows that $2r > d_{\pi}(r) + a_{\pi}(r)$ and
\begin{equation*}
     {d_{\pi}(r) + a_{\pi}(r) \choose r-1} \geq  {d_{\pi}(r) + a_{\pi}(r) \choose r}
\end{equation*}
for all $r \geq i+2$. This shows that the summation in the above inequality is nonnegative.

Finally, we also have that $d_{\pi}(i) = d_{\pi}(i+1) = i-1$ and $a_{\pi}(i) = a_{\pi}(i+1)$, so by Pascal's Identity,
\begin{align*}
    |\SC_{132}^{-1}(\pi')| - |\SC_{132}^{-1}(\pi)| &\geq {d_{\pi}(i) + a_{\pi}(i) + 1 \choose i} - {d_{\pi}(i) + a_{\pi}(i) \choose i} - {d_{\pi}(i+1) + a_{\pi}(i+1) \choose i+1} \\
    &= {i + a_{\pi}(i) \choose i} -  {i + a_{\pi}(i) \choose i+1} \geq 0,
\end{align*}
as desired. The last inequality follows from the fact that $a_{\pi}(i) \leq a_{\pi}(i-1) = a < i$.
\end{proof}

If we start with a permutation $\pi\in\Av(132,213)$ whose corresponding word over $\{A,D\}$ is not of the form $AD \cdots D A \cdots A$ (i.e., the permutation is not simply a decreasing sequence followed by an increasing sequence), then the previous lemma allows us to alter the permutation and obtain a new permutation $\pi'$ with at least as many preimages under $\SC_{132}$ as $\pi$. This is stated in the following corollary. It is well known that a permutation can be written as a decreasing sequence followed by an increasing sequence if and only if it avoids $132$ and $231$ classically. Therefore, the set of permutations in $\Av(132,213)$ that can be written as a decreasing sequence followed by an increasing sequence is precisely $\Av(132,213,231)$. 

\begin{corollary} \label{cor: 132maxpre}
For every positive integer $n$, we have 
\[\max_{\pi\in\Av_n(132,213,231)}|\SC_{132}^{-1}(\pi)| = \max_{\pi\in S_n}|\SC_{132}^{-1}(\pi)|. \]
\end{corollary}
\begin{proof}
Let $\displaystyle M=\max_{\pi\in S_n}|\SC_{132}^{-1}(\pi)|$. By Corollary~\ref{lm: aug1}, there is a permutation $\pi\in \Av_n(132,213)$ such that $|\SC_{132}^{-1}(\pi)|=M$. Suppose $\pi$ contains the pattern $231$. We can alter $\pi$ using Lemma~\ref{lm: DtoA} in order to obtain a new permutation $\pi'$ with $|\SC_{132}^{-1}(\pi)|=M$. Let $b$ and $b'$ be the words over $\{A,D\}$ corresponding to $\pi$ and $\pi'$ respectively. Given a word $w$ over $\{A,D\}$, let $f_1(w)$ denote the length of the longest consecutive string of $D$'s in $w$ starting in the second position of $w$, and let $f_2(w)$ be the length of the longest suffix  of $w$ that uses only the letter $A$. Because $\pi$ contains $231$, it is straightforward to check (using the definition of $b'$ given in the proof of Lemma~\ref{lm: DtoA}) that $f_1(b)\leq f_1(b')$ and $f_2(b)\leq f_2(b')$, where at least one of these inequalities must be strict. This shows that if we repeatedly use Lemma~\ref{lm: DtoA} to alter permutations, we must eventually reach a permutation $\tau\in\Av_n(132,213,231)$ with $|\SC_{132}^{-1}(\tau)|=M$. 
\end{proof}

\begin{theorem} The maximum number of preimages under $\SC_{132}$ or $\SC_{312}$ that a permutation in $S_n$ can have is given by
\[\max_{\pi\in S_n}|\SC_{132}^{-1}(\pi)|=\max_{\pi\in S_n}|\SC_{312}^{-1}(\pi)|={n-1\choose\left\lfloor\frac{n-1}{2}\right\rfloor}.\]
\end{theorem}
\begin{proof}
By complementing, it suffices to consider only the map $\SC_{132}$. Let $\displaystyle M=\max_{\pi\in S_n}|\SC_{132}^{-1}(\pi)|$. By Corollary \ref{cor: 132maxpre}, there exists $\pi\in \Av_n(132,213,231)$ with $|\SC_{132}^{-1}(\pi)|=M$. Since $\pi$ avoids $132$, $213$, and $231$, it must be of the form \[n(n-1)\cdots(n-r+1)123\cdots (n-r).\] The word over $\{A,D\}$ corresponding to $\pi$ is $A\underbrace{D\cdots D}_{r}\underbrace{A\cdots A}_{n-r-1}$. Observe that $d_\pi(k)=\min(k-1,r)$ for all $1\leq k\leq n$. Furthermore, we have $a_\pi(k)=n-r-1$ if $1\leq k\leq r$ and $a_\pi(k)=0$ if $r+1\leq k\leq n$. By Theorem~\ref{th: layer preim} and the Hockey Stick Identity, we have
\begin{equation*}
    M=|\SC_{132}^{-1}(\pi)| = 1+{n-r-1\choose 1}+{n-r\choose 2}+\cdots+{n-2\choose r}={n-1 \choose r} \leq {n-1\choose\left\lfloor\frac{n-1}{2}\right\rfloor}.
\end{equation*}
Moreover, equality is achieved when $r = \left\lfloor\frac{n-1}{2}\right\rfloor$; this completes the proof. 
\end{proof}

\section{The Maps $\SC_{213}$ and $\SC_{231}$}\label{Sec:213and231}

We begin this section with a brief discussion of the sets $\Sort(\SC_{213})$ and $\Sort(\SC_{231})$. 
The initial terms of the sequence $(|\Sort_n(\SC_{213})|)_{n\geq 0}$ are \[1, 1, 2, 5, 15, 50, 180, 686, 2731, 11254.\] This sequence appears to be new.
On the other hand, the first $10$ terms of $(|\Sort_n(\SC_{213})|)_{n\geq 0}$ are \[ 1, 1, 2, 6, 21, 79, 311, 1265, 5275, 22431.\] These numbers match the initial terms in the OEIS sequence A033321, which is the binomial transform of Fine's sequence. Fine's sequence $(F_k)_{k\geq 0}$ can be defined via its generating function \[\sum_{k\geq 0}F_kx^k=\frac{1-\sqrt{1-4x}}{3-\sqrt{1-4x}}.\]

\begin{conjecture}\label{Conj:231Sortable}
For each positive integer $n$, we have \[|\Sort_n(\SC_{231})|=\sum_{k=0}^n\binom{n}{k}F_{k+1}.\] 
\end{conjecture}

Our main focus in this section will be on the maximum number of preimages a permutation in $S_n$ can have under either $\SC_{213}$ or $\SC_{231}$.

\begin{theorem}
For every $n\geq 2$, we have
\begin{equation*}
    \max_{\pi\in S_n}|\SC_{213}^{-1}(\pi)|=\max_{\pi\in S_n}|\SC_{231}^{-1}(\pi)| = 2^{n-2}.
\end{equation*} 
\end{theorem}

\begin{proof}
The equality $\displaystyle\max_{\pi\in S_n}|\SC_{213}^{-1}(\pi)|=\max_{\pi\in S_n}|\SC_{231}^{-1}(\pi)|$ follows directly from Lemma~\ref{LemComplement}, so we can focus on the map $\SC_{231}$. Now consider the process of sending a permutation $\tau = \tau_1 \cdots \tau_n\in S_n$ through a consecutive-$231$-avoiding stack. Say that the entry $\tau_j$ is \emph{premature} if it is popped out of the stack before all of the entries of $\tau$ enter the stack. The effect of sending $\tau$ through the stack is to shift all of the premature entries to the left while preserving their relative order and reverse the remaining (non-premature) entries.
More formally, if $\{\tau_{j_1}, \cdots, \tau_{j_k}\}$ are the premature entries of $\tau$, then $\SC_{231}(\tau) = \tau_{j_1} \cdots \tau_{j_k} \rev(\tau \setminus \{\tau_{j_1}, \cdots, \tau_{j_k}\})$, where $\tau \setminus \{\tau_{j_1}, \cdots, \tau_{j_k}\}$ is simply the sequence $\tau$ with the terms in $\{\tau_{j_1}, \cdots, \tau_{j_k}\}$ removed.

Now fix $\pi=\pi_1\cdots\pi_n\in S_n$, and consider a preimage $\tau$ of $\pi$ under $\SC_{231}$. Suppose that $\tau_{j_1},\ldots,\tau_{j_k}$ are the premature entries of $\tau$ (with $j_1<\cdots<j_k$). By the previous paragraph, $\tau$ is obtained by inserting the entries $\pi_1,\ldots,\pi_k$ into $\rev(\pi_{k+1}\cdots\pi_n)$ so that the entry in position $j_i$ of $\tau$ is $\pi_i$ for all $i\in[k]$. This shows that $\tau$ is uniquely determined by specifying $\pi$ and the set of indices $\{j_1,\ldots,j_k\}$. Finally, notice that the first and last entries of a permutation cannot be premature. Therefore, the positions of the premature entries of $\tau$ must form a subset of $\{2, \ldots, n-1\}$. It follows that $|\SC_{231}^{-1}(\pi)|\leq 2^{n-2}$

To complete the proof, it suffices to show that $|\SC_{231}^{-1}(n (n-1) \cdots 1)|\geq 2^{n-2}$. For each set $J=\{j_1,\ldots,j_k\}\subseteq\{2,\ldots,n-1\}$ (with $j_1<\cdots<j_k$), let $\tau^{J}$ be the permutation obtained by inserting the entries $n,n-1,\ldots,n-k+1$ into the sequence $123\cdots (n-k)$ so that for each $i\in[k]$, the entry $n+1-i$ is in position $j_i$. For example, if $n=8$ and $J=\{2,3,5\}$, then $\tau^J=18726345$. It is straightforward to check that $\SC_{231}(\tau^J)=n(n-1)\cdots 1$ and that $\tau^J\neq\tau^{J'}$ whenever $J\neq J'$. Consequently, $|\SC_{231}^{-1}(n(n-1)\cdots 1)|\geq 2^{n-2}$. 
\end{proof}

\section{The Maps $\SC_{123}$ and $\SC_{321}$}\label{Sec:123and321}

Our goal in this section is to enumerate the sets $\Sort(\SC_{123})$ and $\Sort(\SC_{321})$, thereby proving analogues of two more of the main theorems from \cite{Cerbai}. Our enumeration of $\Sort(\SC_{321})$ generalizes to $\Sort(\SC_{k(k-1) \cdots 1})$, so we focus on $\Sort(\SC_{k(k-1) \cdots 1})$ first instead.

It follows immediately from Lemma~\ref{Lem1} that $\Sort(\SC_{k(k-1) \cdots 1}) = \Av(132, \underline{12\cdots k})$ for $k\geq 3$. In Proposition~3.2 of \cite{HeinHuang} it was shown that this set is enumerated by the so-called \emph{generalized Motzkin numbers} $M_{k-1,n}$, which can be defined via the formula \[M_{k-1,n}=\frac{1}{n+1}\sum_{j=0}^{\left\lfloor n/k\right\rfloor}(-1)^j\binom{n+1}{j}\binom{2n-jk}{n}.\] For $k= 3,4,5,$ and $6$, these numbers are given by OEIS sequences A001006, A036765, A036766, and A036767 respectively \cite{OEIS}.

\begin{theorem}[\!\!\cite{HeinHuang}]
For $k\geq 3$, the sets $\Sort_n(\SC_{k(k-1) \cdots 1}) = \Av_n(132, \underline{12\cdots k})$ are enumerated by the generalized Motzkin numbers. More precisely, \[|\Sort_n(\SC_{k(k-1)\cdots 1})|=M_{k-1,n}.\]
\end{theorem}

The above theorem completes our enumeration of $\Sort(\SC_{k(k-1) \cdots 1})$. We remark that when $k=3$, the generalized Motzkin numbers $M_{2,n}$ are the same as the classical Motzkin numbers $M_n$. These numbers will also be related to our enumeration of $\Sort(\SC_{123})$.

\begin{corollary} \label{ThmEnum321} 
The set $\Sort(\SC_{321})$ is enumerated by the Motzkin numbers. That is, 
\begin{equation*}
    |\Sort_n(\SC_{321})| = M_n.
\end{equation*}
\end{corollary}

The following is a well-known recurrence for the Motzkin numbers.

\begin{lemma} \label{LemMotzkinRecur}
We have $M_0 = M_1 = 1$ and, for $n\geq 2$,
\begin{equation*}
    M_n = M_{n-1} + \sum_{i=0}^{n-2} M_{i}M_{n-2-i}.
\end{equation*}
\end{lemma}

To begin enumerating $\Sort(\SC_{123})$, we first classify the set. In order to do this, it is helpful to use the notion of a \emph{vincular pattern}. We refer the reader to \cite{Steingrimsson} for a formal treatment of vincular patterns; we will only need to consider the set $\Av(132,\underline{321}4,\underline{421}3,\underline{431}2)$. This is the set of permutations that avoid $132$ and also do not contain any occurrences of the patterns $3214$, $4213$, or $4312$ in which the first $3$ entries in the occurrence of the pattern appear consecutively in the permutation. For example, $51324$ belongs to this set of permutations because, although $5,3,2,4$ form an occurrence of the pattern $4213$, the entries $5,3,2$ do not appear consecutively in $51324$. 
\begin{theorem} \label{ThmSort123}
The set $\Sort(\SC_{123})$ consists of the permutations $\pi = \pi_1 \cdots \pi_n$, for some positive integer $n$, that satisfy the following properties:
\begin{itemize}
    \item $\pi \in \Av(132)$
    \item if $\pi_i\pi_{i+1} \cdots \pi_{j}$ is a descending run of $\pi$ of length at least $3$, then there are no entries bigger than $\pi_{j-1}$ to the right of $\pi_{j-1}$ in $\pi$, and 
    the entries $\pi_{i+1},...,\pi_{j-1}$ are consecutive integers.
\end{itemize}
Alternatively, $\Sort(\SC_{123})=\Av(132,\underline{321}4,\underline{421}3,\underline{431}2)$. 
\begin{proof}
It is straightforward to check that a permutation satisfies the two listed conditions if and only if it is in $\Av(132,\underline{321}4,\underline{421}3,\underline{431}2)$. The proof of the first statement is straightforward once we find an expression for $\SC_{123}(\pi)$. Fortunately, this was done in Lemma~\ref{LemSC321 map} for the map $\SC_{321}$, so we may use Lemma~\ref{LemComplement} to find a similar expression for $\SC_{123}(\pi)$ by complementing. Namely, suppose $\pi = d_1 \cdots d_k$, where each $d_i$ is a descending run. We have $\SC_{123}(\pi) = d_1^{m}\cdots d_k^{m}\rev(d_k^e) \cdots \rev(d_1^e)$, where $d_i^m$ is the subsequence of $d_i$ consisting of the entries that are not the first or last entries in $d_i$ and $d_i^{e}$ is the sequence obtained from $d_i$ by deleting $d_i^m$.

To see that all permutations in $\Sort(\SC_{123})$ must satisfy the two conditions, note first that if $abc$ forms an occurrence of the classical pattern $132$ in $\pi$ and $a$,$b$, and $c$ are in the descending runs $d_i, d_j$, and $d_{\ell}$ respectively, then $c$, the second entry of $\rev(d_j^e)$, and the first entry of $\rev(d_i^e)$ form an occurrence of the classical pattern $231$ in $\SC_{123}(\pi)$. This shows that $\Sort(\SC_{123})\subseteq\Av(132)$. For the first part of the second condition, suppose $\pi_i\pi_{i+1} \cdots \pi_{j}$ is a descending run of length at least $3$ and that $\pi_{\ell}$ is an entry to the right of $\pi_{j}$ that is also the first entry of a different descending run. Then $\pi_{j-1} \pi_{\ell} \pi_{j}$ is a subsequence of $\SC_{123}(\pi)$, and  we certainly have $\pi_j < \pi_{j-1}$. It follows that if $\SC_{123}(\pi) \in \Av(231)$, then $\pi_{\ell} < \pi_{j-1}$. Hence, no entry to the right of $\pi_{j-1}$ in $\pi$ is bigger than $\pi_{j-1}$. Finally, the previous two points imply the second part of the second condition. If the entries $\pi_{i+1},...,\pi_{j-1}$ are not consecutive integers, then there is some entry $\pi_{i'}$ that is not part of the descending run containing these entries and that satisfies $\pi_{i+1} > \pi_{i'} > \pi_{j-1}$. Then either  $\pi_{i'}\pi_i\pi_{i+1}$ is a $132$-subsequence or there is an entry to the right of $\pi_{j-1}$ that is bigger than $\pi_{j-1}$.

To see that every permutation satisfying these conditions is indeed in $\Sort(\SC_{123})$, again decompose $\pi$ into its descending runs and write $\SC_{123}(\pi) =  d_1^{m}\cdots d_k^{m}\rev(d_k^e) \cdots \rev(d_1^e)$. The first condition ensures that $\rev(d_k^e) \cdots \rev(d_1^e)$ does not contain an occurrence of the classical pattern $231$. The first part of the second condition ensures that $d_1^{m}\cdots d_k^{m}$ is a decreasing sequence, so it also does not contain an occurrence of $231$. Finally, the first condition and first part of the second condition ensure that there is no occurrence of the pattern $231$ containing entries from both $d_1^{m}\cdots d_k^{m}$ and $\rev(d_k^e) \cdots \rev(d_1^e)$. Indeed, suppose $abc$ is an occurrence such a pattern. As $d_1^{m}\cdots d_k^{m}$ is a decreasing sequence, the entry $a$ must be in $d_1^{m}\cdots d_k^{m}$ while $b$ and $c$ are in  $\rev(d_k^e) \cdots \rev(d_1^e)$. Let $i,j,\ell$ be the indices such that $a$ is in $d_i^m$, $b$ is in $d_j^e$, and $c$ is in $d_{\ell}^e$. Since $c$ appears to the right of $b$ in $\SC_{123}(\pi)$, we must have $\ell<j$. The first part of the second condition guarantees that $j\leq i$. However, this implies that $cba$ is an occurrence of the classical pattern $132$ in $\pi$, which is a contradiction. 
\end{proof}
\end{theorem}

Before enumerating the permutations in $\Sort(\SC_{123})$, we need one more lemma that counts permutations avoiding $132$ classically and $123$ consecutively. We obtain this result via a bijection between $\Av(132)$ and $\Av(231)$. For our purposes, we only need the existence of this bijection and one of its properties. The bijection is described in detail in Lemma~4.1 of \cite{DefantCatalan}. 

\begin{lemma}[\!\!\cite{DefantCatalan}]
For each $n\geq 1$, there exists a bijection $\swd: \Av_n(132)\to \Av_n(231)$ that preserves the set of descents. That is, an index $i$ is a descent of $\pi$ if and only if it is a descent of $\swd(\pi)$.
\end{lemma}
\begin{corollary} \label{Cor123-321}
For each $n\geq 1$, the sets of permutations $\Av_n(132, \underline{321})$, $\Av_n(231, \underline{321})$, and $\Av_n(132, \underline{123})$ are in bijection. Consequently, $|\Av_n(132, \underline{321})| = M_n$. 
\end{corollary}
\begin{proof}
Using the previous lemma, we obtain a chain of bijections \[\Av_n(132, \underline{321})\xrightarrow{\swd}\Av_n(231, \underline{321})\xrightarrow{\rev}\Av_n(132, \underline{123}).\]
By Theorem~\ref{ThmEnum321},  $|\Av_n(132, \underline{321})| = |\Av_n(132, \underline{123})| = M_n$. 
\end{proof}

We need one final lemma before completing the enumeration of $\Sort(\SC_{123})$.
\begin{lemma} \label{LemEnum123}
The number of permutations in $\Sort_n(\SC_{123})$ beginning with the entry $n$ is $M_{n-1}$.
\end{lemma}
\begin{proof}
We proceed by induction on $n$, noting first that the lemma is trivial if $n \leq 2$. Assume $n\geq 3$. Suppose $\pi\in \Sort_n(\SC_{123})$ starts with the entry $n$, and let $\pi'$ be the permutation obtained by deleting the leading entry $n$ from $\pi$. It follows from Theorem~\ref{ThmSort123} that $\pi'\in \Sort_{n-1}(\SC_{123})$ and that $\pi'$ either starts with the entry $n-1$ or starts with an ascent. On the other hand, if $\tau\in \Sort_{n-1}(\SC_{123})$ starts with the entry $n-1$ or starts with an ascent, then the permutation $n\tau$ is in $\Sort_n(\SC_{123})$. By the induction hypothesis, there are $M_{n-2}$ permutations in $\Sort_{n-1}(\SC_{123})$ that start with $n-1$. 

We now need to determine the number of permutations $\tau\in \Sort_{n-1}(\SC_{123})$ that start with an ascent. Let $i\in\{2,\ldots,n-1\}$ be such that $\tau_i=n-1$. We may write $\tau$ as $A(n-1)B$, where $A$ does not start with a descent (meaning it either starts with an ascent or consists of a single entry). Since $\tau$ must avoid $132$ by Theorem~\ref{ThmSort123}, every entry in $A$ is greater than every entry in $B$, so the entries in $A$ are $n-i+2,\ldots,n-2$ and the entries in $B$ are $1, \ldots, n-i+1$. Since $(n-1)B$ has length $n-i+1$ and has the same relative order as $(n-i+1)B$, the number of possibilities for $(n-1)B$ is the same as the number of possibilities for $(n-i+1)B$. Now $(n-i+1)B$ can be any element of $\Sort_{n-i+1}(\SC_{123})$ that starts with $n-i+1$, so it follows from the induction hypothesis that the number of possibilities for $(n-1)B$ is $M_{n-i}$. Next, for each possible choice of $A$, we can subtract $n-i+1$ from each entry of $A$. Theorem~\ref{ThmSort123} implies that this operation yields a bijection between the set of possible choices for $A$ and the set of permutations in $\Av_{i-1}(132,\underline{321})$ that do not start with a descent. If $i\geq 3$, then since $A$ avoids $132$ and starts with an ascent, the first two entries of $A$ must be consecutive integers. As a result, removing the first entry and standardizing gives a bijection between permutations in $\Av_{i-1}(132,\underline{321})$ that do not start with a descent and permutations in $\Av_{i-2}(132,\underline{321})$. This implies that there are $M_{i-2}$ possibilities for $A$ by Lemma~\ref{Cor123-321}. If $i=2$, then it is also certainly true that the number of choices for $A$ is $M_{i-2}=M_0=1$. 

Putting this all together and invoking the Motzkin number recurrence in Lemma~\ref{LemMotzkinRecur}, we find that the number of permutations in $\Sort_n(\SC_{123})$ beginning with $n$ is \[M_{n-2} + \sum_{i=2}^{n-1}M_{i-2}M_{n-i} = M_{n-2} + \sum_{i=0}^{n-3}M_{i}M_{n-3-i} = M_{n-1}.\qedhere\]
\end{proof}

We can now prove that $\Sort(\SC_{123})$ is enumerated by the first differences of Motzkin numbers, which form the OEIS sequence A002026 \cite{OEIS}.
\begin{theorem}
The permutations in $\Sort(\SC_{123})$ are counted by the first differences of Motzkin numbers. That is, 
\begin{equation*}
    |\Sort_n(\SC_{123})| = M_{n+1} - M_{n}.
\end{equation*}
\end{theorem}
\begin{proof}
Let $X_i^n$ denote the set of permutations in $\Sort_n(\SC_{123})$ in which the $i$th entry is $n$. If $\pi\in X_i^n$, then we may write $\pi = AnB$, where $A$ has length $i-1$ and $B$ has length $n-i$. By Theorem~\ref{ThmSort123}, $AnB$ avoids the pattern $132$, so every entry in $A$ is bigger than every entry in $B$, meaning that $A$ is a permutation of the numbers $n-i+1,\ldots, n-1$ and $B$ is a permutation of $1, \ldots, n-i$. Moreover, the standardization of $A$ can be an arbitrary element of $\Av(132, \underline{321})$, so the number of possibilities for $A$ is $M_{i-1}$ by Corollary~\ref{Cor123-321}. Meanwhile, the standardization of $nB$ can be an arbitrary permutation in $\Sort_{n-i+1}(\SC_{123})$ that starts with $n-i+1$, so by Lemma~\ref{LemEnum123}, there are $M_{n-i}$ possibilities for $nB$. Thus, $|X_i^n| = M_{i-1}M_{n-i}$, and \[|\Sort_n(\SC_{123})| = \sum_{i=1}^{n} M_{i-1}M_{n-i} =  \sum_{i=0}^{n-1} M_{i}M_{n-1-i} = M_{n+1} - M_n.\qedhere\]
\end{proof}

\section{Dynamics of the Maps $s_{132}$ and $s_{312}$}\label{Sec:Dynamicsofs132}

The set $\Sort(s_{132})$ was the primary focus of the recent article \cite{Cerbai2}, where it was shown that $\Sort(s_{132})$ is equal to the set of permutations avoiding the classical pattern $2314$ as well as a certain mesh pattern. The same article proved that $|\Sort_n(s_{132})|=\sum_{k=0}^{n-1}\binom{n-1}{k}C_k$, where $C_k$ denotes the $k$th Catalan number. In this section, we briefly discuss some dynamical properties of the map $s_{132}$ and, by complementation, the map $s_{312}$. 

\begin{theorem}\label{Thm:Classical132Periodic}
The periodic points of the map $s_{132}:S_n\to S_n$ are precisely the permutations in $\Av_n(132,231)$. The periodic points of the map $s_{312}:S_n\to S_n$ are precisely the permutations in $\Av_n(213,312)$. When $n\geq 2$, each of these periodic points has period $2$. Furthermore, for every permutation $\pi\in S_n$, we have $s_{132}^{n-1}(\pi)\in\Av_n(132,231)$ and $s_{312}^{n-1}(\pi)\in\Av_n(213,312)$.  
\end{theorem}

\begin{proof}
By Lemma~\ref{LemComplement}, it suffices to prove the desired properties of $s_{132}$. Note that if we have $\pi \in \Av(132,231)$, then $s_{132}(\pi) = \rev(\pi) \in \Av(132, 231)$, so $s_{132}^2(\pi) = \pi$. Hence, $\pi$  is a periodic point of $s_{132}$, and the period is $2$ if $n\geq 2$. 

It remains to prove that $s_{132}^{n-1}(\pi)\in\Av(132,231)$ for every $\pi\in S_n$. This is trivial if $n=1$, so we may assume $n\geq 2$ and induct on $n$. Consider sending a permutation $\tau\in S_n$ through a classical-$132$-avoiding stack. When the entry $1$ enters the stack, the entries below it in the stack must appear in increasing order from top to bottom. Furthermore, the entry $1$ will not leave the stack until after all entries have entered the stack. This implies that the last ascending run of $s_{132}(\tau)$ begins with the entry $1$. Therefore, it suffices to prove that if $\pi\in S_n$ is a permutation whose last ascending run begins with $1$, then $s_{132}^{n-2}(\pi)\in\Av(132,231)$. We prove this by induction on $n$, noting first that it is trivial if $n=2$. Now suppose $n\geq 3$.   

For each $\tau\in S_n$, let $\tau^*$ be the permutation in $S_{n-1}$ obtained by deleting the entry $1$ from $\tau$ and then decreasing all remaining entries by $1$. Observe that if the entries $1$ and $2$ appear consecutively in $\tau$, then they also appear consecutively in $s_{132}(\tau)$ and that $(s_{132}(\tau))^*=s_{132}(\tau^*)$. By induction, we have $s_{132}^{n-2}(\tau^*)\in\Av(132,231)$. If the last ascending run of $\tau$ begins with $1$, then the above analysis of the classical-$132$-avoiding stack shows that the entries $1$ and $2$ appear consecutively in $s_{132}(\tau)$. It then follows that the entries $1$ and $2$ appear consecutively in $s_{132}^{n-2}(\tau)$ and that $(s_{132}^{n-2}(\tau))^*=s_{132}^{n-2}(\tau^*)\in\Av(132,231)$. This implies that $s_{132}^{n-2}(\tau)\in\Av(132,231)$, as desired. 
\end{proof}

Whenever one is confronted with a noninvertible finite dynamical system, it is natural to ask for the maximum number of iterations of the map needed to send every point to a periodic point. For example, Knuth \cite[pages 106--110]{Knuth2} studied this problem for the classical bubble sort map. For West's stack-sorting map $s$, permutations requiring close to the maximum number of iterations to get sorted into the identity were studied by West and Claesson--Dukes--Steingr\'imsson\cite{Claessonn-4, West}. Ungar gave a rather nontrivial argument proving that the maximum number of iterations needed to sort a permutation in $S_n$ using the pop-stack-sorting map is $n-1$ \cite{Ungar}, and the permutations requiring $n-1$ iterations were investigated further by Asinowski, Banderier, and Hackl \cite{Asinowski}. Some other papers that have studied similar questions for other finite dynamical systems include 
\cite{Toom, Hobby, Bentz, Igusa, Etienne, Griggs}.

Our second result in this section will determine the maximum number of iterations of $s_{132}$ (equivalently, $s_{312}$) needed to send a permutation to a periodic point. In what follows, we let $\sd_{132}(\pi)$ denote the smallest nonnegative integer $t$ such that $s_{132}^t(\pi)\in\Av(132,231)$. Note that $\sd_{132}(\pi)$ is also the smallest nonnegative integer $t$ such that $s_{312}^t(\comp(\pi))\in\Av(213,312)$. 

\begin{theorem}\label{Thm:MaxIterations}
For $n\geq 3$, we have \[\max_{\pi\in S_n}\sd_{132}(\pi)=n-1.\] 
\end{theorem}

\begin{proof}
Theorem~\ref{Thm:Classical132Periodic} tells us that $\displaystyle\max_{\pi\in S_n}\sd_{132}(\pi)\leq n-1$, so it suffices to find a permutation $\pi\in S_n$ such that $\sd_{132}(\pi)\geq n-1$. Let $n=3m+r$, where $r\in\{0,1,2\}$. Let $\lambda_m$ denote the permutation in $S_{n-r}$ obtained by taking the skew sum of $m$ copies of the permutation $132$. In other words, $\lambda_m=\xi_m\xi_{m-1}\cdots\xi_1$, where $\xi_j$ denotes the sequence $(3j-2)(3j)(3j-1)$. Let \[\pi=\begin{cases} \lambda_m & \mbox{if } r=0; \\  n\lambda_m & \mbox{if } r=1; \\  (n-1)n\lambda_m & \mbox{if } r=2. \end{cases}\] We claim that $s_{132}^{n-2}(\pi)\not\in\Av(132,231)$; that is, $\sd_{132}(\pi)\geq n-1$. For the sake of simplicity, we will prove this in the case $r=0$; the proofs in the cases $r=1$ and $r=2$ are similar. Thus, $n=3m$ and $\pi=\lambda_m$. If $n=3$, then we are done because $s_{132}(\pi)=s_{132}(132)=231\not\in\Av(132,231)$. Therefore, we may assume $n=3m\geq 6$.

It will be helpful to introduce a little more notation. If $k\geq 1$ is odd (respectively, even), we let $V_k$ denote the permutation in $S_{3k}$ obtained by listing the odd (respectively, even) numbers in $[3k+1]$ in decreasing order and then listing the even (respectively, odd) numbers in $[3k+1]$ in increasing order. For example, we have $V_1=3124$, $V_2=6421357$, and $V_3=9\,7\,5\,3\,1\,2\,4\,6\,8\,10$. For $i\in\{0,1,2\}$, let $D_{k,m}^{(i)}$ be the sequence obtained by writing the elements of $\{3k+2,3k+3,\ldots,3m\}$ that are congruent to $i$ modulo $3$ in decreasing order. For example, $D_{2,5}^{(0)}=15\,12\,9$, $D_{2,5}^{(1)}=13\,10$, and $D_{2,5}^{(2)}=14\,11\,8$. The following two claims can be verified in a straightforward manner using nothing more than the definition of the map $s_{132}$; we leave this verification to the reader. 

\noindent{\bf Claim 1:} We have $s_{132}^4(\lambda_m)=D_{1,m}^{(2)}D_{1,m}^{(0)}V_1\rev\left(D_{1,m}^{(1)}\right)$.

\noindent{\bf Claim 2:} For $1\leq k\leq m-2$, we have \[s_{132}^3\left(D_{k,m}^{(2)}D_{k,m}^{(0)}V_k\rev\left(D_{k,m}^{(1)}\right)\right)=D_{k+1,m}^{(2)}D_{k+1,m}^{(0)}V_{k+1}\rev\left(D_{k+1,m}^{(1)}\right).\] 

These two claims imply that $s_{132}^{n-2}(\pi)=s_{132}^{3m-2}(\lambda_m)=D_{m-1,m}^{(2)}D_{m-1,m}^{(0)}V_{m-1}\rev\left(D_{m-1,m}^{(1)}\right)$. Note that $D_{m-1,m}^{(2)}$ consists of the single entry $3m-1$. Similarly, $D_{m-1,m}^{(0)}$ consists of the single entry $3m$. Finally, $D_{m-1,m}^{(1)}$ is empty. Thus, $s_{132}^{n-2}(\pi)=(3m-1)(3m)V_{m-1}\not\in\Av(132,231)$. 
\end{proof}

\section{Future Directions}\label{Sec:Conclusion}

We have initiated the study of the consecutive-pattern-avoiding stack-sorting maps $\SC_{\sigma}$ and analyzed them as both dynamical systems and sorting procedures. These maps are generalizations of West's stack-sorting map \cite{West} and are variants of the classical-pattern-avoiding stack-sorting maps studied by Cerbai, Claesson, and Ferrari in \cite{Cerbai}. 

Our main results on $\SC_{\sigma}$ as a sorting procedure were the characterization of when $\Sort(\SC_{\sigma})$ is a permutation class and the enumeration of $\Sort(\SC_{\sigma})$ summarized in the table below.

\setlength\belowrulesep{0pt}
\setlength\aboverulesep{0pt}
\begin{center}
\footnotesize{
\begin{tabularx}{.98\textwidth}{X|XXXXXXXXXXX}
 $\sigma$ $\backslash$ $n$ & 0 & 1 & 2 & 3 & 4 & 5 & 6 & 7 & 8 & 9 & OEIS \\
 \hline  \\[-0.25cm]
 $123$ & 1 & 1 & 2 & 5 & 12 & 30 & 76 & 196 & 512 & 1353 & A002006\\ [0.1cm] 
 $132$ &  1 & 1 & 2 & 5 & 14 & 42 & 132 & 429 & 1430 & 4862 & A000108\\ [0.1cm]
 $213$ &  1 & 1 & 2 & 5 & 15 & 50 & 180 & 686 & 2731 & 11254 & unknown \\ [0.1cm]
 $231$ &  1 & 1 & 2 & 6 & 21 & 79 & 311 & 1265 & 5275 & 22431 & unknown\\ [0.1cm]
 $312$ &  1 & 1 & 2 & 5 & 15 & 50 & 179 & 675 & 2649 & 10734 & unknown \\ [0.1cm]
 $321$ &  1 & 1 & 2 & 4 & 9 & 21 & 51 & 127 & 323 & 835 & A001006\\ [0.1cm]
\end{tabularx}
}
\end{center}
It would be interesting to have nontrivial information about the asymptotic behavior of the unknown sequences in this table. Recall that we also have Conjecture~\ref{Conj:231Sortable}, which states that $\Sort(\SC_{231})$ is enumerated by the binomial transform of Fine's sequence (OEIS sequence A033321). 

From the dynamical point of view, we first proved Theorem~\ref{thm: periodicmaintheorem}, which characterized the periodic points of $\SC_\sigma$ for each $\sigma\in S_3$. This theorem has motivated us to formulate the following conjecture, which we have additionally confirmed when $k=4$ and $n\leq 8$.

\begin{conjecture}\label{Conj:PeriodicMain}
Let $\sigma\in S_k$ for some $k\geq 3$. The periodic points of the map $\SC_\sigma:S_n\to S_n$ are precisely the permutations in $\Av_n\left(\underline{\sigma},\underline{\rev(\sigma)}\right)$.
\end{conjecture}

We then asked for the maximum value of $|\SC_{\sigma}^{-1}(\pi)|$ for $\pi\in S_n$. The following table summarizes what we know about these values; note that by Lemma~\ref{LemComplement}, we only need to consider one pattern from each of the pairs $\{123,321\}$, $\{132,312\}$, $\{213,231\}$. 
\begin{center}
\footnotesize{
\begin{tabularx}{\textwidth}{ X|XXXXXXXXXX} 
 $\sigma$ $\backslash$ $n$  & 1 & 2 & 3 & 4 & 5 & 6 & 7 & 8 & 9 & OEIS\\
 \hline \\[-0.25cm]
 $123$ & 1 & 1 & 2 & 3 & 4 & 7 & 11 & 16 & 26 & unknown\\ [0.1cm]
 $132$ & 1 & 1 & 2 & 3 & 6 & 10 & 20 & 35 & 70 & A001405\\ [0.1cm]
 $231$ & 1 & 1 & 2 & 4 & 8 & 16 & 32 & 64 & 128 & A011782\\ [0.1cm]
\end{tabularx}
}
\end{center}
It would be interesting to have nontrivial estimates for the numbers $\displaystyle\max_{\pi\in S_n}|\SC_{123}^{-1}(\pi)|$ appearing in the first row of the above table. 

In Section~\ref{Sec:Dynamicsofs132}, we studied the maximum number of iterations of $s_{132}$ (equivalently, $s_{312}$) needed to send a permutation to a periodic point. This line of investigation is completely open for the consecutive-pattern-avoiding stack-sorting maps. However, we do have the following conjecture, which we have verified for $n\leq 9$. Recall that the set of periodic points of $\SC_{231}:S_n\to S_n$ is $\Av_n(\underline{132},\underline{231})=\Av_n(132,231)$.

\begin{conjecture}
Let $n\geq 3$. We have $\SC_{231}^{2n-4}(\pi)\in\Av_n(132,231)$ for every $\pi\in S_n$. Furthermore, there exists a permutation $\tau\in S_n$ such that $\SC_{231}^{2n-5}(\tau)\not\in\Av_n(132,231)$. 
\end{conjecture}

In \cite{DefantFertility}, the first author showed that the numbers $3,7,11,15,19,23$ are \emph{infertility numbers}, meaning that there does not exist a permutation $\pi$ such that $|s^{-1}(\pi)|\in\{3,7,11,15,19,23\}$. By contrast, we have the following conjecture for the maps $\SC_{\sigma}$ with $\sigma\in S_3$. 

\begin{conjecture}
For every $\sigma\in S_3$ and every positive integer $f$, there exists a permutation $\pi$ such that $|\SC_{\sigma}^{-1}(\pi)|=f$.  
\end{conjecture}

Section~\ref{Sec:Dynamicsofs132} contains the first theorems of a dynamical nature concerning maps of the form $s_{\sigma}$. It would be interesting to have other dynamical results for these maps, including analogues of Theorems~\ref{Thm:Classical132Periodic} and \ref{Thm:MaxIterations} for the maps $s_\sigma$ with $\sigma\in\{123,213,231,321\}$. In Theorems~\ref{Thm:Classical132Periodic} and \ref{Thm:MaxIterations}, we proved that every permutation $\pi\in S_n$ satisfies $s_{132}^{n-1}(\pi)\in\Av(132,231)$ and that there exists at least one permutation $\tau\in S_n$ such that $s_{132}^{n-2}(\tau)\not\in\Av(132,231)$. It would be interesting to obtain nontrivial information about these permutations $\tau$. We also have the following conjecture. If $n$ is odd (respectively, even), let ${\bf V}_n$ denote the permutation in $S_n$ obtained by listing the odd (respectively, even) elements of $[n]$ in decreasing order and then listing the even (respectively, odd) elements of $[n]$ in increasing order. For example, ${\bf V}_6=642135$ and ${\bf V}_7=7531246$. 

\begin{conjecture}
If $\tau\in S_n$ is such that $s_{132}^{n-2}(\tau)\not\in\Av(132,231)$, then $s_{132}^{n-1}(\tau)={\bf V}_n$. 
\end{conjecture} 

For example, one can check that if $\tau\in S_6$ is such that $s_{132}^4(\tau)\not\in\Av(132,231)$, then $s_{132}^5(\tau)={\bf V}_6=642135$. 

Of course, it would also be very natural to consider the maps of the form $\SC_{\sigma^{(1)},\ldots,\sigma^{(r)}}$, where $\SC_{\sigma^{(1)},\ldots,\sigma^{(r)}}$ is defined using a stack whose contents must consecutively avoid all of the patterns $\sigma^{(1)},\ldots,\sigma^{(r)}$. It would also be exciting to have results about the maps $\SC_\sigma$ for patterns $\sigma$ of length at least $4$ (such as progress toward Conjecture~\ref{Conj:PeriodicMain}).  

\section{Acknowledgments}
This research was conducted as part of the University of Minnesota Duluth Mathematics REU and was supported by NSF-DMS
Grant 1949884 and NSA Grant H98230-20-1-0009.
The authors thank Joe Gallian for running the REU program and providing encouragement. We also thank Giulio Cerbai for very helpful discussions. The first author was supported by a Fannie and John Hertz Foundation Fellowship and an NSF Graduate Research Fellowship.

\end{document}